\documentclass[12pt]{amsart}
\usepackage{amsmath, amssymb}
\usepackage{mathrsfs}
\newcommand{\no}[1]{#1}
\renewcommand{\no}[1]{}
\no{\usepackage{times}\usepackage[subscriptcorrection, slantedGreek, nofontinfo]{mtpro}
\renewcommand{\Delta}{\upDelta}}
\usepackage{color}
\usepackage{enumerate,comment}
\usepackage{amssymb}
\usepackage{amssymb}
\usepackage{graphicx}
\usepackage{dsfont}
\usepackage{epstopdf}
\graphicspath{{./pics/}}
\usepackage{booktabs}
\usepackage{threeparttable}
\usepackage{ulem} 

\usepackage{graphicx,color}
\usepackage[colorlinks=true]{hyperref}
\usepackage{tikz}  

\date{\today}
\newcommand{\bel}{\begin{equation} \label}
\newcommand{\ee}{\end{equation}}
\def\beq{\begin{equation}}
\def\eeq{\end{equation}}
\newcommand{\bea}{\begin{eqnarray}}
\newcommand{\eea}{\end{eqnarray}}
\newcommand{\beas}{\begin{eqnarray*}}
\newcommand{\eeas}{\end{eqnarray*}}

\newcommand{\R}{\mathbb{R}}

\newtheorem{theorem}{Theorem}[section]

\newtheorem{corollary}[theorem]{Corollary}
\newtheorem{proposition}[theorem]{Proposition}

\numberwithin{equation}{section}

\allowdisplaybreaks
\providecommand{\abs}[1]{\left\lvert#1\right\rvert}
\providecommand{\norm}[1]{\left\lVert#1\right\rVert}

\def\phi {\varphi}

\title[Determination of source term for Stokes systems]
{Determination of time dependent source terms for Stokes systems in  unbounded domains}
\author{{Adel Blouza, L\'eo Glangetas, Yavar Kian, Van-Sang Ngo}}
\address{Univ Rouen Normandie, CNRS, Normandie Univ, LMRS UMR 6085, F-76000 Rouen, France}
\email{adel.blouza@univ-rouen.fr}
\email{leo.glangetas@univ-rouen.fr} 
\email{yavar.kian@univ-rouen.fr}
\email{van-sang.ngo@univ-rouen.fr}

\begin{document}
\begin{abstract} 
This article is devoted to the analysis of inverse source problems for Stokes systems in unbounded domains where  the corresponding velocity flow is observed on a surface. 
Our main objective  is to study the unique determination of general class of time-dependent and vector-valued source terms with potentially unknown divergence. 
Taking into account the challenges inherent in  this class of inverse source problems, we aim to identify the most precise conditions that ensure their  resolution. 
Motivated by various fluid motion problems,  we explore several class of boundary measurements. 
Our proofs are based on different arguments,  including unique continuation properties   for Stokes systems, approximate controllability, complex analysis, and the application of explicit harmonic functions and explicit solutions for Stokes systems. 
This analysis is complemented by a reconstruction algorithm and   examples of numerical computations.  
\end{abstract}

\maketitle

\section{Introduction}

In this article, we consider the following Stokes system
\begin{equation}\label{eq1}
\begin{cases}
\partial_tu -\nu\Delta u+\nabla p =  F(x,t) & \mbox{in }\R^n\times(0,T),\\
 \nabla\cdot u= 0 & \mbox{in } \R^n\times(0,T), \\
u(x,0)=u_0(x) & x\in\R^n,
\end{cases}
\end{equation}
where the constant $\nu>0$ denotes the viscosity of the fluid, $p$ the pressure, the time $T>0$ fixed, the space dimension $n\geq2$ and where the external force $F\in L^2(0,T;L^2(\R^n))^n$ and the initial data $u_0\in L^2(\R^n)^n$ satisfy $$\nabla\cdot F\in L^2(0,T;L^2(\R^n)),\ \nabla\cdot u_0\equiv 0.$$ The main goal of our work is to study inverse source problems for general class of time dependent source terms in an unbounded domain and to provide general conditions which guarantee the resolution of such inverse problems. Our analysis will be complemented with numerical computation based on Tikhonov regularization algorithm stated in Section~5.

The system \eqref{eq1} describes the motion of low-speed Newtonian fluids, such as gases or liquids, generated by the  external source $F$ as well as the initial velocity flow $u_0$ and propagating in the whole space $\R^n$. Such model can be used for describing different physical phenomenon associated with the weather, ocean currents in the context of different applications including the design of aircraft and cars,  the design of power stations and the analysis of pollution. The goal of our inverse problem analysis consists in detecting the unknown source $F$ from the knowledge of the flow velocity $u$, measured on a surface $S$ located outside the support of $F$. Besides, we will take into account the evolution in time of the source $F$, which can have significant practical applications. Note also that the  knowledge of the source $F$ implies the  knowledge of the velocity flow $u$ everywhere. In that sense, our inverse problem can be equivalently formulated as the determination of the velocity flow $u$ everywhere from its knowledge on the surface $S$.

In this paper, we assume that $u_0$ is given and supp$(F)\subset \Omega\times[0,T]$ where $\Omega$ is a bounded open and simply connected set of $\R^n$ with Lipschitz boundary. Along with the set $\Omega$, we consider an open set $\mathcal O$ with Lipschitz boundary lying in $\R^n\setminus\bar{\Omega}$ (see Figure \ref{fig1}). We introduce the following space of compactly supported smooth functions 
$$\Upsilon:=\{f\in C^\infty_0(\R^n)^n:\ \nabla\cdot f\equiv0\}$$
and denote by $\mathbb H$ (\textit{resp.} $\mathbb V$) the closure of $\Upsilon$ in $L^2(\R^n)^n$ (\textit{resp.} in $H^1(\R^n)^n$) endowed  with the scalar product of $L^2(\R^n)^n$ (\textit{resp.} of $H^1(\R^n)^n$). We also denote by $\mathbb V'$ the dual space of $\mathbb V$.
It is well known (see for instance \cite[Theorem 1.1, Chapter III]{Te} and \cite[Proposition 1.1 and 1.2, Chapter I]{Te}) that problem \eqref{eq1} admits a unique solution $u\in L^2(0,T;\mathbb V)\cap H^1(0,T;\mathbb V')$ with $p\in L^2(0,T;L^2_{loc}(\R^n))$. In our work, we study the unique determination problem of some class of time dependent source $F$ from the knowledge of 
$$u(x,t),\quad x\in S,\ t\in(0,T),$$
when $S=\partial\Omega$ or $S=\partial\mathcal O$.

\begin{figure}
\centering  
\begin{tikzpicture}

\draw[thick,  black!10!red] (0,0) ellipse (3.0cm and 2.0cm);
\draw (0,1.7) node[right, black!10!red] {$\Omega$};

\draw[thick, black!10!blue] (0,0) ellipse (2cm and 1cm);
\draw (0.2,0.0) node[black!10!blue] {$\text{supp} F$};

\draw[thick, black!40!green] (6,1) ellipse (1cm and .5cm);
\draw (6,1) node[black!40!green] {$\mathcal{O}$};

\end{tikzpicture}
\caption{The sets $\Omega$ and $\mathcal{O}$. \label{fig1}} 
\end{figure}
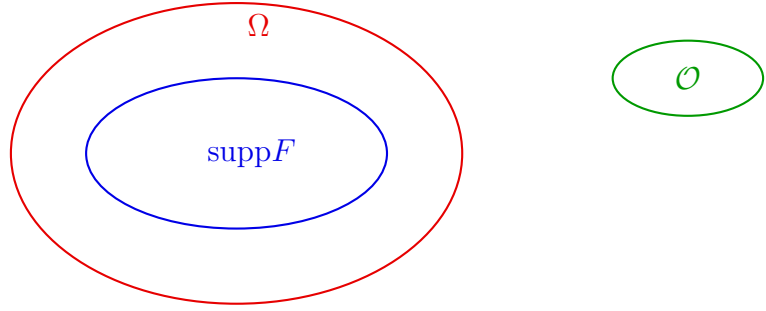

It is well known (see the Appendix for more details of the counter-examples) that it is impossible to solve this inverse problem for general time dependent source term $F$. For this purpose we restrict our analysis to two class of separated variable source terms. Namely, we consider the source terms of the following forms

-  Space-time separated variable case : 
\bel{source1} \quad F(x,t)=\sigma(t)f(x),\quad (x,t)\in\R^n\times(0,T),\ee

-  Multi-layer separated variable case : 
\bel{source2} F(x',x_n,t)=h(x_n)(G(x',t)^T,0)^T,\quad (x',x_n,t)\in\R^{n-1}\times\R\times(0,T),\ee
with $\sigma\in L^2(0,T)$, $f\in L^2(\R^n)^n$, $h\in L^2(\R)$, $G\in L^2(\R^{n-1}\times(0,T))^{n-1}$. Note that sources of the form \eqref{source2} will only be considered for $n\geq3$.

Despite  their interest and physical relevance, inverse problems and more broadly  identification problems associated with motion of fluids are not yet well understood. Nevertheless, several works have been devoted to the study of different identification problems including the detection of an object immersed in a fluid \cite{ACFK,CCOR}, the determination of viscosity from boundary measurements \cite{LUW,LW} or the inverse source problems. We recall that inverse source problems can be distinguished by their important applications in different scientific and mathematical areas (see e.g. \cite{I} for an overview).  In the context of motion of fluids, most of the existing mathematical literature about inverse source problems have been stated for a linearized Navier-Stokes or Stokes systems in a bounded domain  and their results were based on applications of Carleman estimates in the spirit of the seminal work of \cite{BK}. In that category, we can mention the works of \cite{CIPY, GMO,ILY} where the determination of a time-independent and divergence-free source was proved for some class of internal data including the full knowledge of the solution at one fixed time. Such approach was extended by \cite{GT} to the stable determination of similar class of source terms from boundary measurements.

All the above mentioned results considered the determination of a divergence free (or real valued) and time independent source term on a bounded domain. In \cite{M} the author studied the determination of a general class of source terms without imposing any condition on its divergence for steady state  Stokes equations (see also \cite{IY1} for the study of this problem with the extra knowledge of pressure $p$) and the work of \cite{IY2} concerned the determination of a specific class of time dependent and divergence free source terms. As far as we know, a general analysis of determination of time dependent source terms, similar to what has already been considered for other class of partial differential equations (see for instance \cite{KLY,KSXY}), is still missing for Stokes systems. Moreover, in contrast to other classes of systems of partial differential equations (see e.g. \cite{AET,HK,HKLZ,Z}), inverse source problems in an unbounded domain for Stokes systems, which appears naturally for different physical phenomenon related to weather or ocean currents have not been investigated in the mathematical literature.

This article is organized as follows. In Section \ref{sec2}, we introduce the main results, in Section \ref{spacetime} we prove Theorem \ref{t1} and Corollary \ref{c1} while Section \ref{multilayer} will be devoted to the proof of Theorem \ref{t2}. These theoretical results are complemented by a reconstruction algorithm stated in Section 5 and corresponding examples of numerical computation in Section 6. Finally, in the appendix, we recall some properties of the analyticity of the solutions of Stokes systems and we introduce general class of counter-examples for the determination of time dependent source terms, including sources of the form \eqref{source1}-\eqref{source2}, from our measurements.

\section{Main results} \label{sec2}

In this section, we introduce our main results, which consist in the unique determination of general classes of time-dependent source terms of the form \eqref{source1} and \eqref{source2}. We will also give comments about the novelty of our results, the technique used in the proofs and some possible extensions.

We start by considering source terms of the form \eqref{source1}.

\begin{theorem}\label{t1} 
 We consider two space-time separated variable source terms $F_j$ for $j=1,2$ given by \eqref{source1}  
\[ \quad F_j(x,t)=\sigma_j(t)f_j(x),\quad (x,t)\in\R^n\times(0,T),\]
where $\sigma_j\in L^2(0,T)$ with $\sigma_j\not\equiv0$, $f_j\in L^2(\R^n)^n$ with $\nabla\cdot f_j\in L^2(\R^n)$, $f_j\not\equiv0$ 
and supp$(f_j)\subset \Omega$.
Let us consider $u_j$ the solution of \eqref{eq1} with $F=F_j$ 
and assume 
also that one of the following conditions:\\
(i) $\nabla\cdot f_1=\nabla\cdot f_2$ and there exists a  function $\psi$ harmonic in $\R^n$ such that
\bel{t1aa}\int_{\R^n}f_1(x)\cdot\nabla \psi(x)dx\neq0,\ee
(ii) there exist  $T_1\in (0,T]$, $\sigma\in L^2(0,T_1)$ with $\sigma\not\equiv0$, $\mu\in L^2(\R)$, $g_j\in L^2(\R^{n-1})$, $j=1,2$, such that
$$\sigma_1|_{(0,T_1)}=\sigma_2|_{(0,T_1)}=\sigma\not\equiv0,$$
\bel{t1ab}\nabla\cdot f_j(x',x_n)=g_j(x')\mu(x_n),\quad x'\in\R^{n-1},\ x_n\in\R,\ j=1,2,\ee
is fulfilled.

Then the condition 
\bel{t1a} u_1(x,t)=u_2(x,t),\quad x\in S,\ t\in(0,T),\ee
with $S=\partial\Omega$ or $S=\partial\mathcal O$ implies that
\bel{t1b} \sigma_1(t)=\sigma_2(t),\quad f_1(x)=f_2(x),\quad x\in \R^n,\ t\in(0,T),\ee
holds true.
\end{theorem}
We can extend this result to the following partial data result when $\partial\Omega$ and $\partial\mathcal O$ are analytic manifolds.

\begin{corollary}\label{c1} Let the first conditions of Theorem \ref{t1} as well as condition (ii) of Theorem \ref{t1} be fulfilled.
Assume that $\partial\Omega$ and $\partial\mathcal O$ are analytic manifolds and that $\Gamma_1$ and $\Gamma_2$ are open non-empty subsets of $\partial\Omega$ and of $\partial\mathcal O$ respectively. Then the condition \eqref{t1a} with $S=\Gamma_1$ or $S=\Gamma_2$ implies \eqref{t1b}.\end{corollary}

Next, we state our results for source terms of the form \eqref{source2}. From now on, for any $G=(g_1,\ldots,g_{n-1})^T\in L^2(\R^{n-1}\times(0,T))^{n-1}$, we denote by $\nabla'\cdot G \in D'(\R^{n-1} \times (0,T))$ the distribution defined by 
$$\nabla'\cdot G=\sum_{j=1}^{n-1}\partial_{x_j}g_j.$$
\begin{theorem}\label{t2} Let $n\geq3$. 
 We consider two multi-layer separated variable source terms $F_j$ for $j=1,2$ given by \eqref{source2}  
\[ F_j(x',x_n,t)=h_j(x_n)(G_j(x',t)^T,0)^T,\quad (x',x_n,t)\in\R^{n-1}\times\R\times(0,T),\]
where $G_j\in L^2(\R^{n-1}\times(0,T))^{n-1}$ with $\nabla'\cdot G_j\in L^2(\R^{n-1}\times(0,T))$, $h_j\in L^2(\R)$ with $h_j\not\equiv0$. 
We assume that supp$(F_j)\subset \Omega\times[0,T]$ and consider $u_j$ the unique solution of \eqref{eq1} with $F=F_j$. 
 We assume also that there exists $\delta\in(0,T)$ such that
\bel{t2a} G_1(x',t)=G_2(x',t),\quad x'\in\R^{n-1},\ t\in(T-\delta,T).\ee
Let one of the following condition be fulfilled:\\
(i) $h_1=h_2=h$\\
(ii) $\nabla'\cdot G_1=\nabla'\cdot G_2\not\equiv0$.\\
Then, condition \eqref{t1a} implies that 
\bel{t2b} G_1(x',t)=G_2(x',t),\quad h_1(x_n)=h_2(x_n),\quad  x'\in\R^{n-1},\ x_n\in\R,\ t\in(0,T).\ee

\end{theorem}

We emphasize that, according to the classical counter-examples in the appendix, it is in general impossible to determine sources of separated variables $F$ of the form  \eqref{source1}-\eqref{source2}. In that sense, to the best of our knowledge, Theorems \ref{t1} and \ref{t2} seem to provide the most general sufficient conditions which guarantee the determination of this class of source terms from boundary measurements. Our assumptions cover different situations where we have either some knowledge of the divergence-free part of the source $F$ or of its divergence. The novelty of our results comes from the fact that even if our measurements do not involve the pressure $p$ in \eqref{eq1}, we are still able to recover some class of sources $F$ without knowing their divergence. Compared to what was proved for other class of partial differential equations including parabolic equations \cite{KLY,KSXY}, we remark that in this paper, the obstruction to the unique determination of a source $F$ (as described in the appendix) is no longer valid and one can fully determine such class of source terms, due to the extra assumptions (i) in Theorem \ref{t1} and (ii) in Theorem \ref{t2}.  

The proofs of our main results combine arguments applicable to general classes of evolution partial differential equations with specific properties of solutions to Stokes systems. For instance, for Stokes systems, we can determine, up to a constant, the pressure $p$ appearing in \eqref{eq1} on $(\R^n\setminus\bar{\Omega})\times (0,T)$ from the knowledge of the velocity $u$ on $S\times (0,T)$. Then, using this knowledge of $p$, we can extract information about the divergence of the source $F$. Such arguments, combined with the unique continuation properties of Stokes systems, the properties of harmonic functions and the application of the Titchmarsh convolution theorem, allow to prove Theorem \ref{t1}. In order to prove Theorem \ref{t2}, we use the explicit solutions of Stokes systems as well as arguments of complex analysis and approximate controllability.

Let us now give some comments on the surface $S$ where the measurements are taken. The surface $S$ can be either the boundary of a domain $\Omega$ containing the support of the source term or the boundary of a domain $\mathcal O$ away from the support of the source (see Figure 1 for more details). The second situation can potentially have interesting applications and, to the best of our knowledge, it has not been analyzed in any previous work on inverse source problems for evolution partial differential equations in an unbounded domain. In addition, by exploiting the analyticity of the solution of Stokes systems, Corollary \ref{c1} provides an important extension of our results, in which we can restrict our measurements to an arbitrary portion of the domain $\Omega$ or $\mathcal O$, when the boundary of these domains is analytic.

Our theoretical results are complemented by a reconstruction algorithm in Section 5 and corresponding examples of numerical computation in Section 6. Our approach is based on Tikhonov regularization method that can be compared with what was already considered in other works on different classes of partial differential equations in a bounded domain (see for example \cite{JLY,JLLY}). Nevertheless, we need to adjust the algorithm and the numerical computation to Stokes system and take into account the extension of such arguments to unbounded domain. This can be done by using the fact that the knowledge of the solution $u$ of \eqref{eq1} on $\partial\Omega\times (0,T)$ implies its knowledge on $\R^n\setminus\bar{\Omega}\times (0,T)$.

\section{ Space-time separated variable sources} \label{spacetime} 
\subsection{Proof of Theorem \ref{t1}}
We consider $u_j\in L^2(0,T;\mathbb V)\cap H^1(0,T;\mathbb V')$ and $p_j\in L^2(0,T;L^2_{loc}(\R^n))$ solving the Stokes system
\begin{equation}\label{eq2}
\begin{cases}
\partial_tu_j -\nu\Delta u_j+\nabla p_j =  F_j(x,t) & \mbox{in }\R^n\times(0,T),\\
 \nabla\cdot u_j=0 & \mbox{in } \R^n\times(0,T), \\
u_j(x,0)=u_0(x) & x\in\R^n,
\end{cases}
\end{equation}
with 
$F_j(x,t)=\sigma_j(t)f_j(x)$.\par 
We recall that such solution $u_j$ is unique. We assume that condition \eqref{t1a} is fulfilled and we will prove that this condition implies \eqref{t1b}. We decompose our proof in four steps.

\textbf{Step 1.} In this step we will show that
\bel{t1c}u_1(x,t)=u_2(x,t),\quad x\in\R^n\setminus\bar{\Omega},\ t\in(0,T)\ee
and $p_j$, $j=1,2$, can be chosen in such a way that
\bel{t1d}p_1(x,t)=p_2(x,t),\quad x\in\R^n\setminus\bar{\Omega},\ t\in(0,T).\ee
We fix $u=u_1-u_2$, $p=p_1-p_2$ and we notice that $(u,p)$ satisfies the following conditions
\begin{equation}\label{eq3}
\begin{cases}
\partial_tu -\nu\Delta u+\nabla p =  F_1(x,t)-F_2(x,t) & \mbox{in }\R^n\times(0,T),\\
 \nabla\cdot u= 0 & \mbox{in } \R^n\times(0,T), \\
u(x,0)=0 & x\in \R^n,\\
u(x,t)=0 &(x,t)\in S\times(0,T).
\end{cases}
\end{equation}

$\bullet$ First case : We assume first that condition \eqref{t1a} is fulfilled with $S=\partial\Omega$. 
Fixing $v=u|_{\R^n\setminus\bar{\Omega}\times(0,T)}$ and recalling that 
$\text{supp}(F_j)\cap \R^n\setminus\bar{\Omega}\times(0,T)=\emptyset$,  
we see that $v$ solves the following conditions
\begin{equation}\label{eq4}
\begin{cases}
\partial_tv -\nu\Delta v+\nabla p =  0 & \mbox{in }\R^n\setminus\bar{\Omega}\times(0,T),\\
 \nabla\cdot v= 0 & \mbox{in } \R^n\setminus\bar{\Omega}\times(0,T), \\
v(x,0)=0, & x\in \R^n\setminus\bar{\Omega},\\
v(x,t)=0, &(x,t)\in\partial\Omega\times(0,T).
\end{cases}
\end{equation}
Notice that, since $u\in L^2(0,T;\mathbb V)$, $\nabla p\in L^2(0,T;H^{-1}(\R^n))$, $v=u|_{\R^n\setminus\bar{\Omega}\times(0,T)}\in L^2(0,T;H^1_0(\R^n\setminus\bar{\Omega}))$ and $u|_{\Omega\times(0,T)}\in L^2(0,T;H^1_0(\Omega))$,  we have
$$\int_0^t\left\langle \nabla p(\cdot,s), u(\cdot,s)\right\rangle_{H^{-1}(\R^n)^n,H^{1}(\R^n)^n}ds=0,\quad t\in(0,T)$$
which implies that, for all $t\in(0,T)$, we obtain
$$\int_0^t\left\langle \nabla p(\cdot,s)|_{\R^n\setminus\bar{\Omega}}, v(\cdot,s)\right\rangle_{H^{-1}(\R^n\setminus\bar{\Omega})^n,H^{1}_0(\R^n\setminus\bar{\Omega})^n}ds=-\int_0^t\left\langle \nabla p(\cdot,s)|_{\Omega}, u(\cdot,s)|_{\Omega}\right\rangle_{H^{-1}(\Omega)^n,H^{1}_0(\Omega)^n}ds.$$
Moreover, since $\Omega$ is a bounded open Lipschitz domain, $p\in L^2(0,T;L^2_{loc}(\R^n))$ and $u|_{\Omega\times(0,T)}\in L^2(0,T;H^1_0(\Omega)^n)$ with $\nabla\cdot u\equiv0$, applying \cite[Proposition 1.1, Chapter I]{Te} and \cite[Theorem 1.6, Chapter I]{Te}, we obtain
$$\int_0^t\left\langle \nabla p(\cdot,s)|_{\Omega}, u(\cdot,s)|_{\Omega}\right\rangle_{H^{-1}(\Omega)^n,H^{1}_0(\Omega)^n}ds=0,\quad t\in(0,T)$$
which implies that
$$\int_0^t\left\langle \nabla p(\cdot,s)|_{\R^n\setminus\bar{\Omega}},v(\cdot,s)\right\rangle_{H^{-1}(\R^n\setminus\bar{\Omega})^n,H^{1}_0(\R^n\setminus\bar{\Omega})^n}ds=0,\quad t\in(0,T).$$
Thus, taking the scalar product of \eqref{eq4} with $v$ and integrating by parts we obtain
$$\begin{aligned} &\frac{1}{2}\norm{v(\cdot,t)}_{L^2(\R^n\setminus\bar{\Omega})^n}^2+\nu \sum_{j=1}^n\int_0^t\norm{\partial_{x_j}v(\cdot,s)}_{L^2(\R^n\setminus\bar{\Omega})^n}^2ds\\
&=\int_0^t\left\langle  (\partial_sv -\nu\Delta v), v \right\rangle_{H^{-1}(\R^n\setminus\bar{\Omega})^n,H^{1}_0(\R^n\setminus\bar{\Omega})^n}ds\\
&=\int_0^t\left\langle  (\partial_sv -\nu\Delta v+\nabla p), v \right\rangle_{H^{-1}(\R^n\setminus\bar{\Omega})^n,H^{1}_0(\R^n\setminus\bar{\Omega})^n}ds=0,\quad t\in(0,T).\end{aligned}$$
It follows that $u|_{\R^n\setminus\bar{\Omega}\times(0,T)}=v\equiv0$ and condition \eqref{t1c} is fulfilled. In addition, we have 
$$\nabla p|_{\R^n\setminus\bar{\Omega}\times(0,T)}\equiv0,$$
and recalling that $\R^n\setminus\bar{\Omega}$ is connected since $\Omega$ is simply connected, we deduce that there exists $\phi\in L^2(0,T)$ such that
$$p_1(x,t)=p_2(x,t)+\phi(t),\quad x\in\R^n\setminus\bar{\Omega},\ t\in(0,T).$$
Replacing $p_1(x,t)$ by $p_1(x,t)-\phi(t)$, we may assume without loss of generality that \eqref{t1d} is fulfilled.

$\bullet$ Second case : Now let us assume that condition \eqref{t1a} is fulfilled with $S=\partial\mathcal O$. Fixing $w=u|_{\mathcal O\times(0,T)}$, we deduce that $w\in L^2(0,T;H^1(\mathcal O)^n)\cap H^1(0,T;H^{-1}(\mathcal O)^n)$ solves the following  initial boundary value problem
\begin{equation}
\begin{cases}
\partial_tw -\nu\Delta w+\nabla p =  0 & \mbox{in }\mathcal O\times(0,T),\\
 \nabla\cdot w= 0 & \mbox{in } \mathcal O\times(0,T), \\
w(x,0)=0 & x\in \mathcal O,\\
w(x,t)=0 &(x,t)\in\partial\mathcal O\times(0,T).
\end{cases}
\end{equation}
Recalling that  $\mathcal O$ is a bounded open Lipschitz domain and applying \cite[Theorem 1.6, Chapter I]{Te}, we are in position to apply the uniqueness of solutions of \eqref{eq1}, stated in \cite[Theorem 1.1, Chapter III]{Te}, to obtain $w\equiv0$. It follows that
$$u(x,t)=w(x,t)=0,\quad x\in\mathcal O,\ t\in(0,T).$$
Combining this with the fact that $\R^n\setminus\bar{\Omega}$ is connected, $\mathcal O\subset \R^n\setminus\bar{\Omega}$, $u\in L^2(0,T;H^1(\R^n\setminus\bar{\Omega}))^n$, $p\in L^2(0,T;L^2_{loc}(\R^n\setminus\bar{\Omega}))$ satisfy
$$
\begin{cases}
\partial_tu -\nu\Delta u+\nabla p=0& \mbox{in }\R^n\setminus\bar{\Omega}\times(0,T),\\
 \nabla\cdot u= 0& \mbox{in } \R^n\setminus\bar{\Omega}\times(0,T),
\end{cases}
 $$
and applying \cite[Theorem 1]{FL}, we deduce that $u|_{\R^n\setminus\bar{\Omega}\times(0,T)}\equiv0$. Thus condition \eqref{t1c} is fulfilled and, repeating the above arguments, we may assume without loss of generality that \eqref{t1d} is also fulfilled.
\par
\textbf{Step 2.} In this step, we will prove that condition (i) or (ii) of Theorem \ref{t1} implies that there exists $T_\star\in(0,T]$ such that
\bel{t1e} \nabla\cdot f_1=\nabla\cdot f_2\quad\textrm{and}\quad \sigma_1|_{(0,T_\star)}=\sigma_2|_{(0,T_\star)}\not\equiv0.\ee

$\bullet$ Case (i) :
Let us first prove this result by assuming condition (i) of Theorem \ref{t1} is fulfilled. 
Applying the divergence to the system \eqref{eq3}, for all $\phi\in C^\infty_0((0,T)\times\R^n)$, we obtain
$$
 \begin{aligned}\int_0^T\int_{\R^n}(\sigma_1(t)-\sigma_2(t))\nabla\cdot f_1(x)\phi(x,t)dxdt&=\int_0^T\int_{\R^n}(\sigma_1(t)\nabla\cdot f_1(x)-\sigma_2(t)\nabla\cdot f_2(x))\phi(x,t)dxdt\\
&=\int_0^T\int_{\R^n}(\nabla\cdot F_1(x,t)-\nabla\cdot F_2(x,t))\phi(x,t)dxdt\\
&=\left\langle \nabla\cdot(\partial_tu -\nu\Delta u+\nabla p),\phi\right\rangle_{D'((0,T)\times\R^n),C^\infty_0((0,T)\times\R^n)}\\
&=\left\langle \partial_t \nabla\cdot u -\nu\Delta \nabla\cdot u+\Delta p),\phi\right\rangle_{D'((0,T)\times\R^n),C^\infty_0((0,T)\times\R^n)}\\
&=\left\langle \Delta p,\phi\right\rangle_{D'((0,T)\times\R^n),C^\infty_0((0,T)\times\R^n)}.\end{aligned}
$$
Combining this with \eqref{t1d}, we deduce that
\bel{t1f}
\begin{cases}
\Delta p =(\sigma_1(t)-\sigma_2(t))\nabla\cdot f_1(x)&\mbox{in }\R^n\times(0,T),\\
 p=0& \mbox{in } \R^n\setminus\bar{\Omega}\times(0,T).
\end{cases}
\ee
In particular, for a.e. $t\in(0,T)$,  $p(\cdot,t)$ can be seen has an element of the set $\mathcal E'(\R^n)$ of compactly supported distributions of $\R^n$. 
In view of \eqref{t1f}, 
for any harmonic function $\psi$ on $\mathbb{R}^n$ and for 
a.e. $t\in(0,T)$, we get
$$\begin{aligned}0=\left\langle p(\cdot,t),\Delta\psi\right\rangle_{\mathcal E'(\R^n),C^\infty(\R^n)}&=\left\langle \Delta p(\cdot,t),\psi\right\rangle_{\mathcal E'(\R^n),C^\infty(\R^n)}\\
&=(\sigma_1(t)-\sigma_2(t))\int_{\R^n}\nabla\cdot f_1(x)\psi(x)dx\\
&=-(\sigma_1(t)-\sigma_2(t))\int_{\R^n}f_1(x)\cdot\nabla \psi(x)dx.\end{aligned}$$
Then, condition \eqref{t1aa} implies that 
$$\sigma_1(t)-\sigma_2(t)=0,\quad \textrm{a.e. }t\in(0,T),$$
which implies that $\sigma_1=\sigma_2$ on $(0,T)$. This proves that \eqref{t1e} is fulfilled with $T_\star=T$.

$\bullet$ Case (ii) :
Now let us assume that condition (ii) of Theorem \ref{t1} is fulfilled. 
It is clear that  condition \eqref{t1ab} is fulfilled with $\mu\equiv0$ 
since we have $\nabla\cdot f_1\equiv0\equiv\nabla\cdot f_2$ 
and condition \eqref{t1e} will be fulfilled with $T_\star=T_1$. 
So let us assume that condition \eqref{t1ab} is fulfilled with $\mu\not\equiv0$.
Then, recalling that 
$$
\nabla\cdot f_j(x)=\mu(x_n)g_j(x'),\quad j=1,2,\ x=(x',x_n)\in\R^{n-1}\times\R,
$$
we have
\bel{t0g}
\begin{cases}
\Delta p 
=\sigma(t)\mu(x_n)(g_1(x')-g_2(x'))\ & x=(x',x_n)\in\R^{n-1}\times\R,\ t\in(0,T_1),\\
 p= 0 & \mbox{in } \R^n\setminus\bar{\Omega}\times(0,T).
 \end{cases}
 \ee
We fix $\tau\in\R$, $\rho\in\R$, $\omega\in\mathbb S^{n-2}$ and we can check that  
$$\Delta(e^{-i\rho x'\cdot\omega-\rho x_n})=0,\quad x'\in\R^{n-1},\ x_n\in\R.$$
Then equation \eqref{t0g} and Fubini's theorem implies
\bel{t1g}\begin{aligned}0&=\left\langle p,e^{-it\tau}\Delta(e^{-i\rho x'\cdot\omega-\rho x_n})\right\rangle_{L^2(0,T;\mathcal E'(\R^n)),L^2(0,T;C^\infty(\R^n))}\\
&=\left\langle \Delta p(\cdot,t),e^{-it\tau}e^{-i\rho x'\cdot\omega-\rho x_n}\right\rangle_{L^2(0,T;\mathcal E'(\R^n)),L^2(0,T;C^\infty(\R^n))}\\
&=\int_0^T\int_\R\int_{\R^{n-1}} \sigma(t)\mu(x_n)(g_1(x')-g_2(x'))e^{-it\tau}e^{-i\rho x'\cdot\omega-\rho x_n}dx'dx_ndt\\
&=\left(\int_0^T\sigma(t)e^{-it\tau}dt\right)\left(\int_\R\mu(x_n)e^{-\rho x_n}dx_n\right)\left(\int_{\R^{n-1}}(g_1(x')-g_2(x'))e^{-i\rho x'\cdot\omega}dx'\right).\end{aligned}\ee
Recalling that $\sigma\not\equiv0$, we can find $\tau\in\R$ such that
$$\int_0^T\sigma(t)e^{-it\tau}dt\neq0.$$
Moreover, the fact that $f_1$ is compactly supported implies that  $\mu\in L^2(\R)$ is compactly supported. Thus, the map
$$\mathbb C\ni z\mapsto \int_\R\mu(x_n)e^{-z x_n}dx_n$$
is holomorphic in $\mathbb C$ and, since $\mu\not\equiv0$, we may find $\rho_1,\rho_2\in\R$, $\rho_1<\rho_2$ such that
$$\int_\R\mu(x_n)e^{-\rho x_n}dx_n\neq0,\quad \rho\in(\rho_1,\rho_2).$$
Combining this with \eqref{t1g}, we deduce that
$$\int_{\R^{n-1}}(g_1(x')-g_2(x'))e^{-i\rho x'\cdot\omega}dx'=0,\quad \rho\in(\rho_1,\rho_2),\ \omega\in\mathbb S^{n-2}.$$
On the other hand, the fact that $f_j$ is compactly supported implies that $g_j\in L^2(\R^{n-1})$ is compactly supported. Thus, the map 
$$\mathbb C\ni z\mapsto \int_{\R^{n-1}}(g_1(x')-g_2(x'))e^{-iz x'\cdot\omega}dx'$$
is holomorphic in $\mathbb C$ and the unique continuation for holomorphic functions implies that
$$\int_{\R^{n-1}}(g_1(x')-g_2(x'))e^{-i\rho x'\cdot\omega}dx'=0,\quad \rho\geq0,\ \omega\in\mathbb S^{n-2}.$$
Then, from the injectivity of the Fourier transform we deduce that $g_1=g_2$ which implies that \eqref{t1e} is fulfilled with $T_\star=T_1$.\par 
\textbf{Step 3.} In this step we will prove that $f_1=f_2$. For this purpose, applying \eqref{t1c}, \eqref{eq3} and \eqref{t1e}, we deduce that, for  $u=u_1-u_2$, $p=p_1-p_2$,   $(u,p)$ satisfies the following conditions
\begin{equation}\label{eq5}
\begin{cases}
\partial_tu -\nu\Delta u+\nabla p =  \sigma_1(t)f(x) & \mbox{in }\R^n\times(0,T_\star),\\
 \nabla\cdot u= 0 & \mbox{in } \R^n\times(0,T_\star), \\
u(x,0)=0 & x\in \R^n,\\
u(x,t)=0 &(x,t)\in\R^n\setminus\bar{\Omega}\times(0,T).
\end{cases}
\end{equation}
with $f=f_1-f_2$ and $\nabla\cdot f=0$. Therefore, $f\in\mathbb H$ and we can fix $y\in  L^2(0,T;\mathbb V)\cap H^1(0,T;\mathbb V')$ and $q\in L^2(0,T;L^2_{loc}(\R^n))$ solving the system 
\begin{equation}\label{eq6}
\begin{cases}
\partial_ty -\nu\Delta y+\nabla q =  0 & \mbox{in }\R^n\times(0,T),\\
 \nabla\cdot y= 0 & \mbox{in } \R^n\times(0,T), \\
y(x,0)=f(x) & x\in \R^n.
\end{cases}
\end{equation}
In light of \eqref{eq5}, one can easily check that
$$u(\cdot,t)=\int_0^t\sigma_1(s)y(\cdot,t-s)ds,\quad t\in (0,T_\star)$$
and $p$ can be chosen in such a way that
$$p(\cdot,t)=\int_0^t\sigma_1(s)q(\cdot,t-s)ds,\quad t\in (0,T_\star).$$
Fix $\phi\in C^\infty_0(\R^n\setminus\bar{\Omega})^n$. 
In light of \eqref{t1c}, $u$ is supported in $\Omega$ and we find
$$\int_0^t\sigma_1(s)\left\langle y(\cdot,t-s),\phi\right\rangle_{L^2(\R^n)^n}ds=\left\langle u(\cdot,t),\phi\right\rangle_{L^2(\R^n)^n}=0,\quad t\in(0,T_\star).$$
Applying the Titchmarsh convolution theorem (see
\cite[Theorem VII]{Ti}) we deduce that there exist $\tau_1,\tau_2\in[0,T_\star]$
satisfying $\tau_1+\tau_2\geq T_\star$ such that
$$\sigma_1(t)=0,\quad t\in(0,\tau_1),$$
$$\left\langle y(\cdot,t),\phi\right\rangle_{L^2(\R^n)^n}=0,\quad t\in(0,\tau_2).$$
Since $\sigma_1|_{(0,T_\star)}\not\equiv0$, we have 
$$\tau_1\leq \tau_3:=\sup\{t\in(0,T_\star):\ \sigma_1|_{(0,t)}\equiv0\} <T_\star$$ 
and it follows that $\tau_2\geq T_\star-\tau_1\geq T-\tau_3>0$. Therefore, we have
$$\left\langle y(\cdot,t),\phi\right\rangle_{L^2(\R^n)^n}=0,\quad t\in(0,T_\star-\tau_3)$$
and, recalling that $\tau_3$ depends only on $\sigma_1$ we deduce that the above identity holds true for any $\phi\in C^\infty_0(\R^n\setminus\bar{\Omega})^n$. Thus, we have
$$y(x,t)=0,\quad x\in \R^n\setminus\bar{\Omega},\quad t\in(0,T_\star-\tau_3),
$$
and combining this with \eqref{eq6} and the unique continuation result of \cite[Theorem 1]{FL}, we deduce that $y=0$ on $\R^n\times(0,T_\star-\tau_3)$. In addition, recalling that $y\in C([0,T_\star];L^2(\R^n)^n)$ (see e.g. \cite[Theorem 1.1, Chapter III]{Te}), we deduce that $f=y(\cdot,0)=0$ in $\R^n$ which implies that $f_1=f_2$.\par 
\textbf{Step 4.} In this step we will complete the proof of Theorem \ref{t1} by showing by contradiction that $\sigma_1=\sigma_2$. For this purpose, let us assume the contrary. We recall that according to Step 3 we have $f_1=f_2$. In light of \cite[Remark 1.5, Chapter I]{Te}, we can find $\psi\in L^2_{loc}(\R^n)$ such that $f_1-\nabla\psi\in\mathbb H$.
By replacing in \eqref{eq1} with $F=F_j$, $j=1,2$,  $p_j$ with $p_j-\sigma_j(t)\psi(x)$, we may assume without loss of generality that $f_1\in \mathbb H$. Then, repeating the arguments of Step 3, we deduce that
$$u(\cdot,t)=\int_0^t(\sigma_1(s)-\sigma_2(s))w_1(\cdot,t-s)ds,\quad t\in (0,T)$$
with $w_1\in  L^2(0,T;\mathbb V)\cap H^1(0,T;\mathbb V')$ and $q\in L^2(0,T;L^2_{loc}(\R^n))$ solving the system 
\begin{equation}\label{eq7}
\begin{cases}
\partial_tw_1 -\nu\Delta w_1+\nabla q =  0 & \mbox{in }\R^n\times(0,T),\\
 \nabla\cdot w_1= 0 & \mbox{in } \R^n\times(0,T), \\
w_1(x,0)=f_1(x) & x\in \R^n.
\end{cases}
\end{equation}
Fixing $\phi\in C^\infty_0(\R^n\setminus\bar{\Omega})^n$ and applying \eqref{t1c}, 
$u$ is supported in $\Omega$ and we find
$$\int_0^t(\sigma_1(s)-\sigma_2(s))\left\langle w_1(\cdot,t-s),\phi\right\rangle_{L^2(\R^n)^n}ds=\left\langle u(\cdot,t),\phi\right\rangle_{L^2(\R^n)^n}=0,\quad t\in(0,T).$$
Applying the Titchmarsh convolution theorem,  we can find $\tau_1,\tau_2\in[0,T]$
satisfying $\tau_1+\tau_2\geq T$ such that
$$\sigma_1(t)-\sigma_2(t)=0,\quad t\in(0,\tau_1),$$
$$\left\langle w_1(\cdot,t),\phi\right\rangle_{L^2(\R^n)^n}=0,\quad t\in(0,\tau_2).$$
Since $\sigma_1\neq\sigma_2$, we have $$\tau_1\leq \tau_3:=\sup\{t\in(0,T):\ \sigma_1-\sigma_2|_{(0,t)}\equiv0\}<T$$ and it follows that $\tau_2\geq T-\tau_1\geq T-\tau_3>0$. Therefore, we have
$$
\left\langle w_1(\cdot,t),\phi\right\rangle_{L^2(\R^n)^n}=0,\quad t\in(0,T-\tau_3)
$$
and recalling that $\tau_3$ is independent of $\phi$ we deduce that the above identity holds true for any $\phi\in C^\infty_0(\R^n\setminus\bar{\Omega})^n$. Thus, we have
$$
w_1(x,t)=0,\quad x\in \R^n\setminus\bar{\Omega},\quad t\in(0,T-\tau_3)
$$
and combining this with \eqref{eq7} and the unique continuation result of \cite[Theorem 1]{FL}, we deduce that $w_1=0$ on $\R^n\times(0,T-\tau_3)$. Thus, we have $f_1=w_1(\cdot,0)=0$ in $\R^n$ which contradicts the fact that $f_1\not\equiv0$. 
Therefore, we have $\sigma_1=\sigma_2$ and \eqref{t1b} is fulfilled. This completes the proof of Theorem \ref{t1}.\par 
\subsection{Proof of  Corollary \ref{c1}}
This subsection is devoted to the proof of Corollary \ref{c1}. In all this subsection, we assume that condition \eqref{t1a} with $S=\Gamma_1$ or $S=\Gamma_2$ is fulfilled and we will show that this condition implies that $f_1=f_2$ and $\sigma_1=\sigma_2$. 
Without loss of generality we assume that condition \eqref{t1a} is fulfilled with $S=\Gamma_1 \subset \partial\Omega$, 
the case $S=\Gamma_2 \subset \partial O$ can be treated in a similar fashion. 
We divide the proof  into two steps. 
Namely, in the first step we prove that $f_1=f_2$ while in the second step we show that $\sigma_1=\sigma_2$.\par 
\textbf{Step 1.} In this step we will prove that $f_1=f_2$. Following the proof of Theorem \ref{t1}, we only need to show that the condition 
\bel{c1a} u_1(x,t)=u_2(x,t),\quad x\in \partial\Omega,\ t\in(0,T_1),\ee
with $T_1$ given by condition (ii) of Theorem \ref{t1}, is fulfilled.  For this purpose, we fix $u=u_1-u_2$,  $f=f_1-f_2$ and we define $\psi\in L^2_{loc}(\R^n)$ such that $f-\nabla\psi\in\mathbb H$. Choosing $p=p_1-p_2-\sigma(t)\psi$ and applying condition (ii) of Theorem \ref{t1}, we notice that $u$ solves the problem
\begin{equation}\label{eq3b}
\begin{cases}
\partial_tu -\nu\Delta u+\nabla p = \sigma(t)(f(x)-\nabla\psi(x))& \mbox{in }\R^n\times(0,T_1),\\
 \nabla\cdot u= 0 & \mbox{in } \R^n\times(0,T_1), \\
u(x,0)=0 & x\in\R^n,\\
u(x,t)=0 & (x,t)\in \Gamma_1\times(0,T_1).  
\end{cases}
\end{equation}
Then,  in a similar way to Theorem \ref{t1}, we can prove that $u$ takes the form
\bel{c1b}u(\cdot,t)=\int_0^t\sigma(s)y(\cdot,t-s)ds,\quad t\in (0,T_1),\ee
with $y$ solving the problem
\begin{equation}\label{eqqq7}
\begin{cases}
\partial_ty -\nu\Delta y+\nabla q =  0 & \mbox{in }\R^n\times(0,\infty),\\
 \nabla\cdot y= 0 & \mbox{in } \R^n\times(0,+\infty), \\
y(x,0)=f(x)-\nabla\psi(x) & x\in \R^n.
\end{cases}
\end{equation}
 We fix $\phi\in C^{\infty}_0(\Gamma_1)^n$ and we notice that
$$\int_0^t\sigma(s)\left\langle y(\cdot,t-s),\phi\right\rangle_{L^2(\partial\Omega)^n}ds=\left\langle u(\cdot,t),\phi\right\rangle_{L^2(\partial\Omega)^n}=0,\quad t\in (0,T_1).$$
Applying the Titchmarsh convolution theorem (see
\cite[Theorem VII]{Ti}) we deduce that there exist $\tau_1,\tau_2\in[0,T_1]$
satisfying $\tau_1+\tau_2\geq T_1$ such that
$$\sigma(t)=0,\quad t\in(0,\tau_1),$$
$$\left\langle y(\cdot,t),\phi\right\rangle_{L^2(\partial\Omega)^n}=0,\quad t\in(0,\tau_2).$$
Since $\sigma|_{(0,T_1)}\not\equiv0$, we have 
$$\tau_1\leq \tau_3:=\sup\{t\in(0,T_1):\ \sigma|_{(0,t)}\equiv0\} <T_1$$ 
and it follows that $\tau_2\geq T_1-\tau_1\geq T_1-\tau_3>0$. Therefore, we have
$$\left\langle y(\cdot,t),\phi\right\rangle_{L^2(\partial\Omega)^n}=0,\quad t\in(0,T_1-\tau_3)$$
and, recalling that $\tau_3$ depends only on $\sigma$ we deduce that the above identity holds true for any $\phi\in C^{\infty}_0(\Gamma_1)^n$. Thus, we have
\bel{c1c}y(x,t)=0,\quad x\in \Gamma_1,\quad t\in(0,T_1-\tau_3).\ee
In addition, according to Proposition \ref{p1} of the Appendix, for all $t\in(0,T_1-\tau_3)$,  the map $\R^n\ni x\mapsto y(x,t)$ is real analytic. Combining this with the fact that $\partial\Omega$ is an analytic manifold, \eqref{c1c} and the unique continuation of real analytic functions implies that 
$$
y(x,t)=0,\quad x\in \partial\Omega,\quad t\in(0,T_1-\tau_3).
$$
Then, using \eqref{c1b} we deduce that \eqref{c1a} is fulfilled and Theorem \ref{t1} implies that $f_1=f_2$.\par
\textbf{Step 2.} In this step we will prove by contradiction that $\sigma_1=\sigma_2$. For this purpose, let us assume the contrary.  We recall that according to Step 1 we have $f_1=f_2\not\equiv0$ and, in a similar way to Step 4 of Theorem \ref{t1},  we may assume without loss of generality that $f_1\in \mathbb H$. Then, repeating the arguments of Theorem \ref{t1}, we deduce that
$$u(\cdot,t)=\int_0^t(\sigma_1(s)-\sigma_2(s))w_1(\cdot,t-s)ds,\quad t\in (0,T),$$
with $w_1\in  L^2(0,T;\mathbb V)\cap H^1(0,T;\mathbb V')$ and $q\in L^2(0,T;L^2_{loc}(\R^n))$ solving the system \eqref{eq7}.
Fixing $\phi\in C^{\infty}_0(\Gamma_1)^n$, we find
$$\int_0^t(\sigma_1(s)-\sigma_2(s))\left\langle w_1(\cdot,t-s),\phi\right\rangle_{L^2(\partial\Omega)^n}ds=\left\langle u(\cdot,t),\phi\right\rangle_{L^2(\partial\Omega)^n}=0,\quad t\in(0,T).$$
Applying the Titchmarsh convolution theorem,  we can find $\tau_1,\tau_2\in[0,T]$
satisfying $\tau_1+\tau_2\geq T$ such that
$$\sigma_1(t)-\sigma_2(t)=0,\quad t\in(0,\tau_1),$$
$$\left\langle w_1(\cdot,t),\phi\right\rangle_{L^2(\partial\Omega)^n}=0,\quad t\in(0,\tau_2).$$
Since $\sigma_1\neq\sigma_2$, we have $$\tau_1\leq \tau_3:=\sup\{t\in(0,T):\ \sigma_1-\sigma_2|_{(0,t)}\equiv0\}<T$$ and it follows that $\tau_2\geq T-\tau_1\geq T-\tau_3>0$. Therefore, we have
$$\left\langle w_1(\cdot,t),\phi\right\rangle_{L^2(\partial\Omega)^n}=0,\quad t\in(0,T-\tau_3)$$
and, recalling that $\tau_3$ is independent of $\phi$ we deduce that the above identity holds true for any $\phi\in C^{\infty}_0(\Gamma_1)^n$. Thus, we have
$$w_1(x,t)=0,\quad x\in \Gamma_1,\quad t\in(0,T-\tau_3).$$
Moreover, applying Proposition \ref{p1} of the Appendix we deduce that, for all $t\in(0,T-\tau_3)$,  the map $\R^n\ni x\mapsto w_1(x,t)$ is analytic. Combining this with the fact that $\partial\Omega$ is an analytic manifold,  the unique continuation of real analytic functions implies that 
\bel{c1d}w_1(x,t)=0,\quad x\in \partial\Omega,\quad t\in(0,T-\tau_3).\ee
Combining this with \eqref{eq7} and the arguments of Step 1 of Theorem \ref{t1}, we deduce that \eqref{c1d} implies that
$$w_1(x,t)=0,\quad x\in \R^n\setminus\bar{\Omega},\quad t\in(0,T-\tau_3)$$
and, combining this with \eqref{eq7} and the unique continuation result of \cite[Theorem 1]{FL}, we deduce that $w_1=0$ on $\R^n\times(0,T-\tau_3)$. Thus, we have $f_1=w_1(\cdot,0)=0$ in $\R^n$ which contradicts the fact that $f_1\not\equiv0$. Therefore, we have $\sigma_1=\sigma_2$ and \eqref{t1b} is fulfilled. This completes the proof of Corollary \ref{c1}.

\section{ Multi-layer separated variable sources} \label{multilayer}
We assume that condition \eqref{t1a} is fulfilled and we will prove that this condition implies \eqref{t2b}. In a similar way to Step 1 of Theorem \ref{t1}, we can see that $u=u_1-u_2$, $p=p_1-p_2$ solve \eqref{eq3}, condition \eqref{t1c} is fulfilled and $p_j$, $j=1,2$, can be chosen in such a way that \eqref{t1d} is fulfilled. Using these properties we will decompose the proof of Theorem \ref{t2} into three steps.

\textbf{Step 1.} We start by proving that if condition (i) of Theorem \ref{t2} is fulfilled, we have $\nabla'\cdot G_1=\nabla'\cdot G_2$. In this step, we fix $h=h_1=h_2$. For this purpose, in a similar way to Step 2 of Theorem \ref{t1} applying the divergence to the system \eqref{eq3} and using \eqref{t1c} and \eqref{t1d}, we deduce that
\bel{t2c}\begin{cases}\Delta p =(\nabla'\cdot G_1(x',t)-\nabla'\cdot G_2(x',t))h(x_n)& \mbox{in }\R^n\times(0,T),\\
 p= 0 & \mbox{in } \R^n\setminus\bar{\Omega}\times(0,T).\end{cases}
 .\ee
Fix $\tau\in\R$, $\rho\in\R$, $\omega\in\mathbb S^{n-2}$. Multiplying \eqref{t2c} by the map
$$e^{-it\tau}e^{-i\rho\omega\cdot x'-\rho x_n},\quad t\in(0,T),\ x'\in\R^n,\ x_n\in\R,$$
 integrating by parts and applying Fubini theorem, we find
$$\begin{aligned}&\left(\int_0^T\int_{\R^{n-1}}(\nabla'\cdot G_1(x',t)-\nabla'\cdot G_2(x',t))e^{-i\rho\omega\cdot x'}e^{-it\tau}dx'dt\right)\left(\int_\R h(x_n)e^{-\rho x_n}dx_n\right)\\
&=\int_0^T\int_{\R^{n-1}}\int_\R(\nabla'\cdot G_1(x',t)-\nabla'\cdot G_2(x',t))h(x_n)e^{-i\rho\omega\cdot x'-\rho x_n}e^{-it\tau}dx'dx_ndt\\
&=\left\langle \Delta p(\cdot,t),e^{-it\tau}e^{-i\rho x'\cdot\omega-\rho x_n}\right\rangle_{L^2(0,T;\mathcal E'(\R^n)),L^2(0,T;C^\infty(\R^n))}\\
&=\left\langle p,e^{-it\tau}\Delta(e^{-i\rho x'\cdot\omega-\rho x_n})\right\rangle_{L^2(0,T;\mathcal E'(\R^n)),L^2(0,T;C^\infty(\R^n))}=0.\end{aligned}$$
Since $F_1$ is compactly supported, we deduce that $h\in L^2(\R)$ is also compactly supported and
the map
$$\mathbb C\ni z\mapsto \int_\R h(x_n)e^{-z x_n}dx_n$$
is holomorphic in $\mathbb C$. Then the injectivity of the Fourier transform combined with the fact that $h\not\equiv 0$ imply that there exist $0<\rho_1<\rho_2$, such that
$$ \int_\R h(x_n)e^{-\rho x_n}dx_n\neq0,\quad \rho\in (\rho_1,\rho_2).$$
Thus, fixing $\tilde{g}=\nabla'\cdot G_1-\nabla'\cdot G_2$ extended by $0$ to $\R^{n-1}\times\R$, we obtain
\bel{t2d}\int_R\int_{\R^{n-1}}\tilde{g}(x',t)e^{-i\rho\omega\cdot x'}e^{-it\tau}dx'dt=0,\quad \tau\in\R,\ \rho\in (\rho_1,\rho_2),\ \omega\in\mathbb S^{n-2}.\ee
In addition, using the fact that $\tilde{g}$ is compactly supported in $\R^{n-1}\times\R$, we deduce that for all $\tau\in\R$ and all $\omega\in\mathbb S^{n-2}$, the map
$$\mathbb C\ni z\mapsto \int_R\int_{\R^{n-1}}\tilde{g}(x',t)e^{-iz\omega\cdot x'}e^{-it\tau}dx'dt$$
is holomorphic in $\mathbb C$. Thus, by unique continuation of holomorphic functions, \eqref{t2d} implies
$$\int_R\int_{\R^{n-1}}\tilde{g}(x',t)e^{-i\rho\omega\cdot x'}e^{-it\tau}dx'dt=0,\quad \tau\in\R,\ \rho>0,\ \omega\in\mathbb S^{n-2}$$
and from the injectivity of the Fourier transform, we deduce that $\tilde{g}\equiv0$. It follows that $\nabla'\cdot G_1=\nabla'\cdot G_2$.
\par
\textbf{Step 2.} In this step we show that if condition (ii) of Theorem \ref{t2} is fulfilled, we have $h_1=h_2$. Repeating the above arguments, applying the divergence to the system \eqref{eq3} and using \eqref{t1c} and \eqref{t1d}, we deduce that
\bel{t22c}\begin{cases}\Delta p =\nabla'\cdot G_1(x',t)(h_1(x_n)-h_2(x_n))& \mbox{in }\R^n\times(0,T),\\
 p= 0 & \mbox{in } \R^n\setminus\bar{\Omega}\times(0,T)\end{cases}.
 \ee
Then, fixing $\tau\in\R$, $\rho>0$ and integrating by parts we obtain
$$\begin{aligned}&\left(\int_0^T\int_{\R^{n-1}}\nabla'\cdot G_1(x',t)e^{-\rho\omega\cdot x'}e^{-it\tau}dx'dt\right)\left(\int_\R (h_1(x_n)-h_2(x_n))e^{-i\rho x_n}dx_n\right)\\
&=\int_0^T\int_{\R^{n-1}}\int_\R \nabla'\cdot G_1(x',t)(h_1(x_n)-h_2(x_n))e^{-\rho\omega\cdot x'-i\rho x_n}e^{-it\tau}dx'dx_ndt\\
&=\left\langle \Delta p(\cdot,t),e^{-it\tau}e^{-\rho x'\cdot\omega-i\rho x_n}\right\rangle_{L^2(0,T;\mathcal E'(\R^n)),L^2(0,T;C^\infty(\R^n))}\\
&=\left\langle p,e^{-it\tau}\Delta(e^{-\rho x'\cdot\omega-i\rho x_n})\right\rangle_{L^2(0,T;\mathcal E'(\R^n)),L^2(0,T;C^\infty(\R^n))}=0.\end{aligned}$$
Since $\nabla'\cdot G_1\not\equiv0$, as a consequence of the injectivity of the Fourier transform, we can find $\tau\in\R$ such that the map
$$g_\tau: x'\mapsto\int_0^T\nabla'\cdot G_1(x',t)e^{-it\tau}dt$$
is not uniformly vanishing. Moreover, since $F_1$ is compactly supported,  $G_1$ will be compactly supported and the same will be true for the map $g_\tau$. Thus, the map
$$z\mapsto \int_{\R^{n-1}}g_\tau(x') e^{-z\omega\cdot x'}dx'$$
with be holomorphic in $\mathbb C$ and non-uniformely vanishing since $g_\tau\not\equiv0$. Therefore, we can find $0<\rho_1<\rho_2$, such that, for all $\rho\in(\rho_1,\rho_2)$, we get
$$\int_0^T\int_{\R^{n-1}}\nabla'\cdot G_1(x',t)e^{-\rho\omega\cdot x'}e^{-it\tau}dx'dt=\int_{\R^{n-1}}g_\tau(x') e^{-\rho\omega\cdot x'}dx'\neq 0.$$
It follows that
$$\int_\R (h_1(x_n)-h_2(x_n))e^{-i\rho x_n}dx_n=0, \quad \rho\in(\rho_1,\rho_2)$$
and, combining as above the fact that $h_1-h_2$ is compactly supported with the injectivity of the Fourier transform, we deduce that this condition implies $h_1-h_2\equiv0$ and $h_1=h_2$.\\
\par
\textbf{Step 3.} In this step we will complete the proof of Theorem \ref{t2}.  From now on, we fix $h_1=h_2=h$ and we denote by $G$ the extension of $G_1-G_2$ by zero to $\R^{n-1}\times\R$. Notice also that according to Step 1, we have
\bel{t2e}\nabla'\cdot G(x',t)=0,\quad x'\in\R^{n-1},\ t\in\R.\ee
Fix $\tau>0$, $\lambda\in(0,1)$, $\omega\in\mathbb S^{n-2}$ and $\omega_j\in\mathbb S^{n-2}$, $j=1,\ldots,n-2$, such that $\{\omega,\omega_1,\ldots,\omega_{n-2}\}$ is an orthonormal basis of $\R^{n-1}$. For $k=1,\ldots,n-2$, we consider the map
\bel{t2f}w_k(x',x_n,t)=\exp\left(-\nu\tau^2t-\tau\sqrt{1-\lambda^2}x'\cdot\omega-\tau\lambda x_n\right) (\omega_k^T,0)^T,\quad (x',x_n,t)\in \R^{n-1}\times\R\times\R\ee
and one can easily check that $w_k\in C^\infty(\R^n\times\R)^n$ satisfies the following conditions
\begin{equation}\label{t2g}
\begin{cases}
-\partial_tw_k -\nu\Delta w_k=  0& \mbox{in }\R^n\times(0,T-\delta),\\
 \nabla\cdot w_k= 0 & \mbox{in } \R^n\times(0,T-\delta).
\end{cases}
\end{equation}
Now let us fix $B_R$ the open ball of center $0$ and of radius $R>0$ with $R$ sufficiently large such that $\bar{\Omega}\subset B_R$ and consider $U$ an open non-empty set of $\R^n$ contained in $B_R\setminus\bar{\Omega}$. We set also the sets
$$\Upsilon(B_R):=\{f\in C^\infty_0(B_R)^n:\ \nabla\cdot f\equiv0\},$$
 $\mathbb V(B_R)$ the closure of $\Upsilon(B_R)$ in $H^1(B_R)^n$ and $\mathbb V(B_R)'$ the dual space of  $\mathbb V(B_R)$.
It is well known that the unique continuation result of \cite[Theorem 1]{FL} implies the approximate controllability of the Stokes system on $B_R$ from the set $U$ (see e.g. \cite{Fc} for more details).  This allows us to show that for any $\varepsilon>0$ there exists $H_{k,\varepsilon}\in L^2(B_R\times(T-\delta,T))$ satisfying supp$(H_{k,\varepsilon})\subset U\times [T-\delta,T]$ such that the solution $v_{k,\varepsilon}\in L^2(T-\delta,T;\mathbb V(B_R))\cap H^1(T-\delta,T;\mathbb V(B_R)')$ of the system
\begin{equation}\label{t2h}
\begin{cases}
-\partial_t v_{k,\varepsilon} -\nu\Delta v_{k,\varepsilon}+\nabla q_{k,\varepsilon}=  H_{k,\varepsilon}& \mbox{in }B_R\times(T-\delta,T),\\
 \nabla\cdot v_{k,\varepsilon}= 0 & \mbox{in } B_R\times(T-\delta,T),\\
v_{k,\varepsilon}(x,t)=0& (x,t)\in\partial B_R\times (T-\delta,T),\\
v_{k,\varepsilon}(x,T)=0& x\in B_R,
\end{cases}
\end{equation}
satisfies the estimate
\bel{t2j}\norm{v_{k,\varepsilon}(\cdot,T-\delta)-w_k(\cdot,T-\delta)|_{B_R}}_{L^2(B_R)^n}\leq \varepsilon.\ee
Note that here we have used the fact that $w_k(\cdot,T-\delta)|_{B_R}\in L^2(B_R)^n$ and $\nabla\cdot w_k(\cdot,T-\delta)\equiv0$.
We fix $k\in\{1,\ldots,n-2\}$, $\varepsilon>0$ and we set $y_{k,\varepsilon}\in L^2(0,T;H^1(B_R))^n$ defined on $B_R\times(0,T) $ by 
$$y_{k,\varepsilon}(\cdot,t)=\left\{\begin{array}{ll} w_k(\cdot,t)&\textrm{ if }t\in(0,T-\delta),\\
v_{k,\varepsilon}(\cdot,t)&\textrm{ if }t\in(T-\delta,T).\end{array}\right.$$
In light of \eqref{t1c} and \cite[Theorem 1.6, Chapter I]{Te}, we find $u|_{B_R\times(T-\delta,T)}\in L^2(T-\delta,T;\mathbb V(B_R))$.
Combining this with \eqref{t2a}, \eqref{t1c}, \eqref{t1d},  \eqref{t2g}, we obtain
$$ \begin{aligned}&\int_0^{T-\delta}\int_{\R^n}(F_1-F_2)\cdot w_kdxdt\\
&=\int_0^T\int_{B_R}(F_1-F_2)\cdot y_{k,\varepsilon}dx dt\\
&=\int_0^{T-\delta}\left\langle \partial_tu-\nu\Delta u+\nabla p, w_k\right\rangle_{\mathcal E'(\R^n)^n, C^\infty(\R^n)^n} dt+\int_{T-\delta}^T\left\langle \partial_tu-\nu\Delta u+\nabla p), v_{k,\varepsilon}\right\rangle_{\mathbb V(B_R)',\mathbb V(B_R)} dt\\
&=\int_0^{T-\delta}\left\langle u, (-\partial_tw_k -\nu\Delta w_k)\right\rangle_{\mathcal E'(\R^n)^n, C^\infty(\R^n)^n} dt+\int_{T-\delta}^T\left\langle  (-\partial_tv_{k,\varepsilon} -\nu\Delta v_{k,\varepsilon}+\nabla q_{k,\varepsilon}),u\right\rangle_{\mathbb V(B_R)',\mathbb V(B_R)}dt\\
&\ \ \ +\int_{B_R} u(x,T-\delta)\cdot(w_k(x,T-\delta)-v_{k,\varepsilon}(x,T-\delta))dx. 
\end{aligned}$$
Taking in account that 
\[ \text{supp}(H_{k,\varepsilon})\subset U\times [T-\delta,T]\subset \R^n\setminus\bar{\Omega}\times[0,T],\] 
the supports of $u$ and $H_{k,\varepsilon}$ are disjoint and we get
$$
\begin{aligned}
&\int_0^{T-\delta}\int_{\R^n}(F_1-F_2)\cdot w_kdxdt\\
&=\int_{T-\delta}^T\int_{B_R}u\cdot H_{k,\varepsilon}dx dt+\int_{B_R} u(x,T-\delta)\cdot(w_k(x,T-\delta)-v_{k,\varepsilon}(x,T-\delta))dx\\
&=\int_{B_R} u(x,T-\delta)\cdot(w_k(x,T-\delta)-v_{k,\varepsilon}(x,T-\delta))dx.\end{aligned}$$
Applying Cauchy-Schwarz inequality and \eqref{t2j}, we get 
$$\begin{aligned}\abs{\int_0^{T-\delta}\int_{\R^n}(F_1-F_2)\cdot w_kdxdt}&\leq \norm{u(\cdot,T-\delta)}_{L^2(B_R)^n}\norm{v_{k,\varepsilon}(\cdot,T-\delta)-w_k(\cdot,T-\delta)}_{L^2(B_R)^n}\\
&\leq \norm{u(\cdot,T-\delta)}_{L^2(B_R)^n}\varepsilon.\end{aligned}$$
Sending $\varepsilon\to0$, we find
$$\int_0^{T-\delta}\int_{\R^n}(F_1-F_2)\cdot w_kdxdt=0,$$
which combined with \eqref{t2a} implies that, for all $\tau>0$, $\lambda\in(0,1)$, we have
\bel{t2k}\begin{aligned}&\left(\int_\R\int_{\R^{n-1}}G(x',t)\cdot\omega_k e^{-\nu\tau^2 t}e^{\tau\sqrt{1-\lambda^2}x'\cdot\omega}dx'dt\right)
\left(\int_\R h(x_n)e^{-\tau\lambda x_n}dx_n\right)\\
&=\int_0^{T-\delta}\int_{\R^n}(F_1-F_2)\cdot w_kdxdt=0.\end{aligned}\ee
 In a similar way to Step 1, we can prove that the map
$$\mathbb C\ni z\mapsto \int_\R h(x_n)e^{-z \tau x_n}dx_n$$
is holomorphic in $\mathbb C$ and non uniformly vanishing. Thus, there exists $\lambda_1,\lambda_2\in(0,1)$, $\lambda_1<\lambda_2$ such that
$$\int_\R h(x_n)e^{-\tau\lambda x_n}dx_n\neq0,\quad \lambda\in (\lambda_1,\lambda_2).$$
Combining this with \eqref{t2k}, we obtain
$$\int_\R\int_{\R^{n-1}}G(x',t)\cdot\omega_k e^{-\nu\tau^2 t}e^{\tau\sqrt{1-\lambda^2}x'\cdot\omega}dx'=0,\quad \tau>0,\  \lambda\in(\lambda_1,\lambda_2),$$
which implies that
\bel{t2l}\int_\R\int_{\R^{n-1}}G(x',t)\cdot\omega_k e^{-\nu\tau^2 t}e^{\mu x'\cdot\omega}dx'=0,\quad \tau>0,\  \mu\in\left(\tau \sqrt{1-\lambda_1^2},\tau \sqrt{1-\lambda_2^2}\right).\ee
Now, recalling that $G$ is compactly supported we deduce that, for all $\tau>0$, the map
$$\mathbb C\ni z\mapsto\int_\R\int_{\R^{n-1}}G(x',t)\cdot\omega_k e^{-\nu\tau^2 t}e^{z x'\cdot\omega}dx'$$
is holomorphic in $\mathbb C$ and condition \eqref{t2l} combined with properties of unique continuation for holomorphic functions implies
\bel{t2m}\int_\R\int_{\R^{n-1}}G(x',t)\cdot\omega_k e^{-\nu\tau^2 t}e^{-i\rho x'\cdot\omega}dx'=0,\quad \tau>0,\  \rho>0,\ k=1,\ldots,n-2.\ee
Moreover, using the fact that $\nabla'\cdot G\equiv0$, we find
$$\int_\R\int_{\R^{n-1}}G(x',t)\cdot\omega e^{-\nu\tau^2 t}e^{-i\rho x'\cdot\omega}dx'=\frac{1}{i\rho} \int_\R\int_{\R^{n-1}}\nabla'\cdot G(x',t) e^{-\nu\tau^2 t}e^{-i\rho x'\cdot\omega}dx'=0.$$
Combining this with the fact that $\{\omega,\omega_1,\ldots,\omega_{n-2}\}$ is an orthonormal basis of $\R^{n-1}$, we obtain
\bel{t2n}\int_\R\int_{\R^{n-1}}G(x',t) e^{-\nu\tau^2 t}e^{-i\rho x'\cdot\omega}dx'=0,\quad \tau>0,\  \rho>0\ee
where we notice that this identity holds true for any $\omega\in\mathbb S^{n-2}$. In the same way, we deduce from \eqref{t2n} that
$$\int_\R\int_{\R^{n-1}}G(x',t) e^{-is t}e^{-i x'\cdot\xi}dx'=0,\quad s\in\R,\  \xi\in\R^{n-1}$$
and the injectivity of the Fourier transform implies that $G\equiv0$. Thus we have $G_1=G_2$ which completes the proof of Theorem \ref{t2}.

\section{Source  reconstruction algorithm}
In this section, we present the numerical reconstruction associated with our theoretical results, focusing specifically on the framework  described in Theorem \ref{t1}.
We consider  $F\in L^2((0,T)\times\R^n)^n$ of the space-time separated form $F(t,x)=\sigma(t)f(x)$  where  $\sigma\in L^2(0,T)$, $\sigma\not\equiv0$,  and $f\in L^2(\R^n)^n$ with supp$(f)\subset\bar{\Omega}$ is an unknown function to reconstruct.  We denote by $u_f$ the solution of \eqref{eq1}. 

First, we observe that for $u_f$,  the solution of the system \eqref{eq1},  fixing $g=u_f|_{\partial\Omega\times(0,T)}$ the map $v=u_f|_{\R^n\setminus\bar{\Omega}\times(0,T)}$ will be the unique solution of the exterior system. 
\begin{equation}\label{eq44}
\begin{cases}
\partial_tv -\nu\Delta v+\nabla p =  0 & \mbox{in }\R^n\setminus\bar{\Omega}\times(0,T),\\
 \nabla\cdot v= 0 & \mbox{in } \R^n\setminus\bar{\Omega}\times(0,T), \\
v(x,0)=0, & x\in \R^n\setminus\bar{\Omega},\\
v(x,t)=g(x,t), &\ (x,t)\in\partial\Omega\times(0,T).
\end{cases}
\end{equation}
The uniqueness of the solution of the system \eqref{eq44} enables us to define the  mapping 
$$
\Lambda: u_f|_{\partial\Omega\times(0,T)}=g\mapsto v=u_f|_{\R^n\setminus\bar{\Omega}\times(0,T)}.
$$
Moreover, under appropriate regularity assumptions, we can establish the continuity of the map $\Lambda$.  Following this observation, we will present  our numerical reconstruction results using  the data $u_f|_{\R^n\setminus\Omega\times(0,T)}$  instead of  $u_f|_{\partial\Omega\times(0,T)}$. 
Suppose  we have access to the noisy contaminated measurement $u^\delta\in L^2(\R^n\setminus\Omega\times(0,T))^n$ with noise level $\delta>0$,  such that  for $f_{true}\in L^2(\R^n)^n$,  the source term to be determined,  we have 
$$
\norm{u_{f_{true}}-u^\delta}_{L^2(\R^n\setminus\bar{\Omega}\times(0,T))^n}\leq\delta.
$$
Then, the numerical reconstruction of the source term $f_{true}$ can be formulated
as a least squares problem with Tikhonov regularization, expressed  in terms of the  minimization of the following functional
$$
J_\lambda(f)=\norm{u_{f}-u^\delta}_{L^2(\R^n\setminus\bar{\Omega}\times(0,T))^n}^2+\lambda\norm{f}_{L^2(\R^n)^n}^2,\quad f\in L^2(\R^n),\ \textrm{supp$(f)\subset\bar{\Omega}$},
$$
with $\lambda>0$ the regularization parameter. By considering the  Fr\'echet derivative of this functional, for any $f,h\in L^2(\R^n)$ supported on $\bar{\Omega}$,  we find
$$
J_\lambda'(f)h=2\int_0^T\int_{\R^n\setminus\bar{\Omega}}(u_f-u^\delta)\cdot u'(f)hdxdt+2\lambda \int_{\R^n}f\cdot hdx.
$$
Furthermore, $v=u'(f)h $, solves the problem
\begin{equation}\label{eqqq2}
\begin{cases}
\partial_tv -\nu\Delta v+\nabla p_1 =  \sigma(t)h(x), & \mbox{in }\R^n\times(0,T),\\
 \nabla\cdot v= 0, & \mbox{on } \R^n\times(0,T), \\
v(x,0)=0, & x\in\R^n,
\end{cases}
\end{equation}
In order to reduce the computational costs in evaluating the Fr\'echet derivatives, we introduce  the adjoint   system of (\ref{eq1}), which is formulated as  the following backward terminal value problem: 
\begin{equation}\label{eq1001}
\begin{cases}
-\partial_tz_f -\nu\Delta z_f+\nabla q =  (u_f(x,t)-u^\delta(x,t))\chi_{\R^n\setminus\bar{\Omega}}(x), & \mbox{in }\R^n\times(0,T),\\
 \nabla\cdot z_f= 0, & \mbox{on } \R^n\times(0,T), \\
z_f(x,T)=0, & x\in \R^n,
\end{cases}
\end{equation}
where $\chi_{\R^n\setminus\bar{\Omega}}$ denotes the charactersitic function of ${\R^n\setminus\bar{\Omega}}$
and $z_f\in L^2(0,T;\mathbb V)\cap H^1(0,T;\mathbb V')$ is the solution of the system. Using $z_f$, we obtain
$$
\begin{aligned}J_\lambda'(f)h&=2\int_0^T\int_{\R^n}(-\partial_tz_f -\nu\Delta z_f+\nabla q)\cdot vdxdt+2\lambda \int_{\R^n}f\cdot hdx\\
&=2\int_0^T\int_{\R^n}z_f\cdot (\partial_tv -\nu\Delta v+\nabla p_1)dxdt+2\lambda \int_{\R^n}f\cdot hdx\\
&=2\int_{\R^n}\left(\int_0^T\sigma(t)z_f(x,t)dt+\lambda f(x)\right)\cdot h(x)dx.\end{aligned}
$$
Therefore, by fixing $c>0$, the equation 
$J_\lambda'(f)=0$ is equivalent to $\lambda f|_{\Omega}=-\int_0^T\sigma(t)z_f(\cdot,t)|_{\Omega}dt$ or, equivalently, to 
\begin{equation}\label{reconstEq}
f|_{\Omega}=\frac{c}{c+\lambda}f|_{\Omega}-\frac{1}{c+\lambda}\int_0^T\sigma(t)z_f(\cdot,t)|_{\Omega}dt.
\end{equation}
Hence, the  reconstruction  of $f_{true}$ is achievable by  solving the equation (\ref{reconstEq}) through an iterative reconstruction algorithm outlined as follows. We start with $f_0\in  L^2(\R^n)^n$ with supp$(f_0)\subset\bar{\Omega}$ and   consider $u_{f_0}$,  the solution of \eqref{eq1} with $ F(x,t)=\sigma(t)f_0(x)$.
Next,  we consider $z_{f_0}$ the solution of \eqref{eq1001} with $f=f_0$. We then define the sequence $(f_k)_{k\in\mathbb N}$ of $L^2(\R^n)^n$,  supported on $\bar{\Omega}$, using the following iterative process: 
\begin{equation}\label{IterativeReconstEq}
f_{k+1}|_{\Omega}=\frac{c}{c+\lambda}f_k|_{\Omega}-\frac{1}{c+\lambda}\int_0^T\sigma(t)z_{f_k}(\cdot,t)|_{\Omega}dt, \; k=0,1,\dots, 
\end{equation}
where $c>0$ is a tuning parameter that balances the influence  between previous and current steps. \par The convergence of the scheme depends critically   on the choice of $c$. \par 

The reconstruction   algorithm is  as follows:
\begin{enumerate}[i.] 
\item  Choose a tolerance $\tau >0$, a regularization  parameter $\lambda>0$, and a tuning constant $c>0$. Provide  an initial guess $f_0\in L^2(\R^n)^n$ with $n=2$ or $n=3$, and set $k=0$. 
\item Compute $f_{k+1}$ using  the iterative update (\ref{IterativeReconstEq}).
\item If $\norm{f_{k+1}-f_k}_{L^2(\Omega)^n}/\norm{f_k}_{L^2(\Omega)^n}\leq \tau
$, stop the iteration. Otherwise, $k\leftarrow k+1$.
\end{enumerate}
\section{Numerical experiments}
In this implementation, we apply  the iterative reconstruction algorithm to approximate the spatial source term $f$ in  the unsteady Stokes system on a bounded domain.  
Let $R>0$ be  a large constant, $\tilde\Omega=]0,R[^2$ be a bounded  set of $\R^2$  containing    $\bar\Omega$ and  $T>0$. Assuming  $F\in L^2(0,T;L^2(\tilde\Omega)^2)$  and 
$u_0\in L^2(\tilde\Omega)^2$ being such that $\nabla\cdot F\in L^2(0,T;L^2(\tilde\Omega))$, $\nabla\cdot u_0\equiv0$. 
To enforce the incompressibility condition  $\nabla \cdot \mathbf{u} = 0$, we use the well-known penalty method for Stokes system.  This approach modifies the Stokes system, making it more tractable numerically. For further details, see \cite{GiraultRaviart}. \par 
Let $0<\varepsilon<1$ be  a penalty parameter and consider the following unsteady  following  equations with homogeneous Dirichlet boundary condition: 
\begin{equation}\label{StokesNumerics}
\begin{cases}
\partial_tu -\nu\Delta u+\nabla p =  F(x,t), & \mbox{in }\tilde\Omega\times(0,T),\\
 \nabla\cdot u= \varepsilon p, & \mbox{in } \tilde\Omega\times(0,T), \\
 u(x,t)=0, & (x,t)\in \partial\tilde\Omega\times(0,T), \\
 u(x,0)=u_0(x), & x\in\tilde\Omega,
\end{cases}
\end{equation}
where $u\in L^2(0,T; H^1_0(\tilde\Omega)^2)$ and  $p\in L^2(0,T; L^2_0(\tilde\Omega))$ with  $L^2_0(\tilde\Omega)$ the space of function  in $L^2(\tilde\Omega)$ with null mean value on $\tilde\Omega$.  We assume here that $u_0$ is known and supp$(F)\subset \Omega\times[0,T]$ with $\Omega$ a bounded open and simply connected set of $\tilde \Omega$ with Lipschitz boundary.  We also  introduce the adjoint system of (\ref{StokesNumerics}) which is formulated by: 
\begin{equation}\label{eq100}
\begin{cases}
-\partial_tz_f -\nu\Delta z_f+\nabla q =  (u_f(x,t)-u^\delta(x,t))\chi_{\tilde\Omega\setminus\bar{\Omega}}(x), & \mbox{in }\tilde\Omega\times(0,T),\\
 \nabla\cdot z_f= \varepsilon q, & \mbox{on } \tilde\Omega\times(0,T), \\
  z_f(x,t)=0, & (x,t)\in \partial\tilde\Omega\times(0,T), \\
z_f(x,T)=0, & x\in \tilde\Omega,
\end{cases}
\end{equation}
where $\chi_{\tilde\Omega\setminus\bar{\Omega}}$ denotes the characteristic function of ${\tilde\Omega\setminus\bar{\Omega}}$
and $z_f\in L^2(0,T;\mathbb V)\cap H^1(0,T;\mathbb V')$ is the solution of the system.\par 
We now consider the weak penalized formulation  of (\ref{StokesNumerics})   which can be stated as follows: \par  Find $(u,p)\in L^2(0,T; H^1_0(\tilde\Omega;\R^2))\times L^2(0,T; L^2_0(\tilde\Omega))$ such that :
\begin{align}\label{WeakStokes1}
\int_{\tilde \Omega}\partial_t u\cdot v\; dx +\nu\int_{\tilde \Omega} \nabla u:\nabla v\; dx +\int_{\tilde \Omega}\nabla p\cdot v &=  \int_{\tilde \Omega} F(x,t)\cdot v\; dx, \forall v\in H^1_0(\tilde\Omega;\R^2)\\ 
\label{WeakStokes2}
 -\int_{\tilde \Omega}\nabla\cdot u\,  q&=  \varepsilon\int_{\tilde \Omega}p q,  \forall q\in L^2(\tilde\Omega). 
\end{align}
It is important to note that, due  to equation (\ref{WeakStokes2}) and the fact that $\varepsilon >0$, the pressure $p$ is uniquely determined. Any additive constant would contradict the relationship  $\nabla u =\varepsilon p$.  \par 
Our objective is to implement the finite element discretization for the penalized formulation of the Stokes problem and its adjoint system. We will then apply the iterative reconstruction process (\ref{IterativeReconstEq}) introduced in the previous section to numerically identify the spatial component of the source term in problem (\ref{StokesNumerics}).\par 

We present three tests. Without loss of generality,  we set 
$$
\tilde\Omega=]0,3[^2, \quad T=1, \quad F(x,t)=\sigma(t) f_{true}(x). 
$$
When the true source terme $f_{true}$ is given,  we produce the noisy observation data  by adding uniform random noises to the true data, {\it i.e.}
$$
u^\delta(x,t)=u_{f_{true}}(x,t)+\delta \mbox{rand}(-1,1)u_{f_{true}}(x,t), (x,t)\in \tilde\Omega\times]0,T[,  
$$
where $\mbox{rand}(-1,1)$ denotes the uniformly distributed random numbers in $[-1,1]$ and $\delta\geq 0$ is the level noise. \par 
Throughout this section, we will fix the known temporal component $\sigma$ in the source term, the regularization parameter $\lambda$, the penalty parameter $\varepsilon$ and the tolerance parameter $\tau$  as 
$$
\sigma(t)=e^t,\quad \lambda=10^{-5}, \quad \varepsilon=10^{-9}, \quad  \mbox{and }  \tau=10^{-3}. 
$$
In addition to illustrative figures,   we  evaluate the numerical performance for different examples by the relative $L^2$-norm error
$$
\textrm{err}:={{\|f_k-f_{true}\|_{L^2}}\over{\|f_{true}\|_{L^2}}}
$$
where $k$ is the iteration number and $f_k$ the reconstructed solution  produced by   (\ref{IterativeReconstEq})  when    (iii) of the reconstruction algorithm  is checked  for the  given tolerance parameter $\tau$. \par The penalized Stokes problem and  its adjoint backward problem are solved by explicit  numerical scheme : finite difference method in time and finite element method in space. As usual, to approximate the unknowns of the unsteady Stokes system using finite element discretization, we should choose a discretization that effectively handles the velocity and pressure fields while satisfying the incompressibility constraint \(\nabla \cdot \mathbf{u} = 0\). A common and effective choice is the $P_1$Bubble-$P_1$ finite element method, which uses  piecewise linear polynomials with bubble functions  for the velocity components and piecewise linear polynomials for the pressure, as introduced by Arnold, Brezzi, and Fortin \cite{ArBrForin}.  This choice ensures that the finite element spaces for the velocity and pressure satisfy the Ladyzhenskaya-Babuska-Brezzi (inf-sup) condition, which is necessary for the stability of the mixed finite element formulation. Additionaly, it  leads  to a relatively low number degrees of freedom while providing a good approximate solution. 
\subsection{Example 1. } $f_{true}(x)=(0.1+x_1/6+x_2/6)  \chi_\Omega (x)$.
We set a time step $\Delta t=0.07$ and the mesh size $h=0.1$. The support of  $f_{true}$ is a subset of $\bar\Omega=[3/4;9/4]$. The initial guess is $f_0(x)=0.8\chi_\Omega (x)$.     \par 
\begin{table}[htp]
\caption{Parameter setting and the corresponding numerical performances in Case 1.}
\begin{center}
\begin{tabular}{|c|c|c|c|}\hline C & k & err & Illustration  \\
\hline 0.01 & 10 & 15\% & \mbox{Figure (a)} \\
\hline 0.01 & 20 & 14,2\% &  \mbox{Figure (b)} \\
\hline 0.01 & 30 & 13,8\% &  \mbox{Figure (c)} \\
\hline
\end{tabular}
\end{center}
\label{Parameter setting and the corresponding numerical performances in Case 1.}
\end{table}
\begin{figure}[htbp!]
\begin{center}
\includegraphics[scale=0.25]{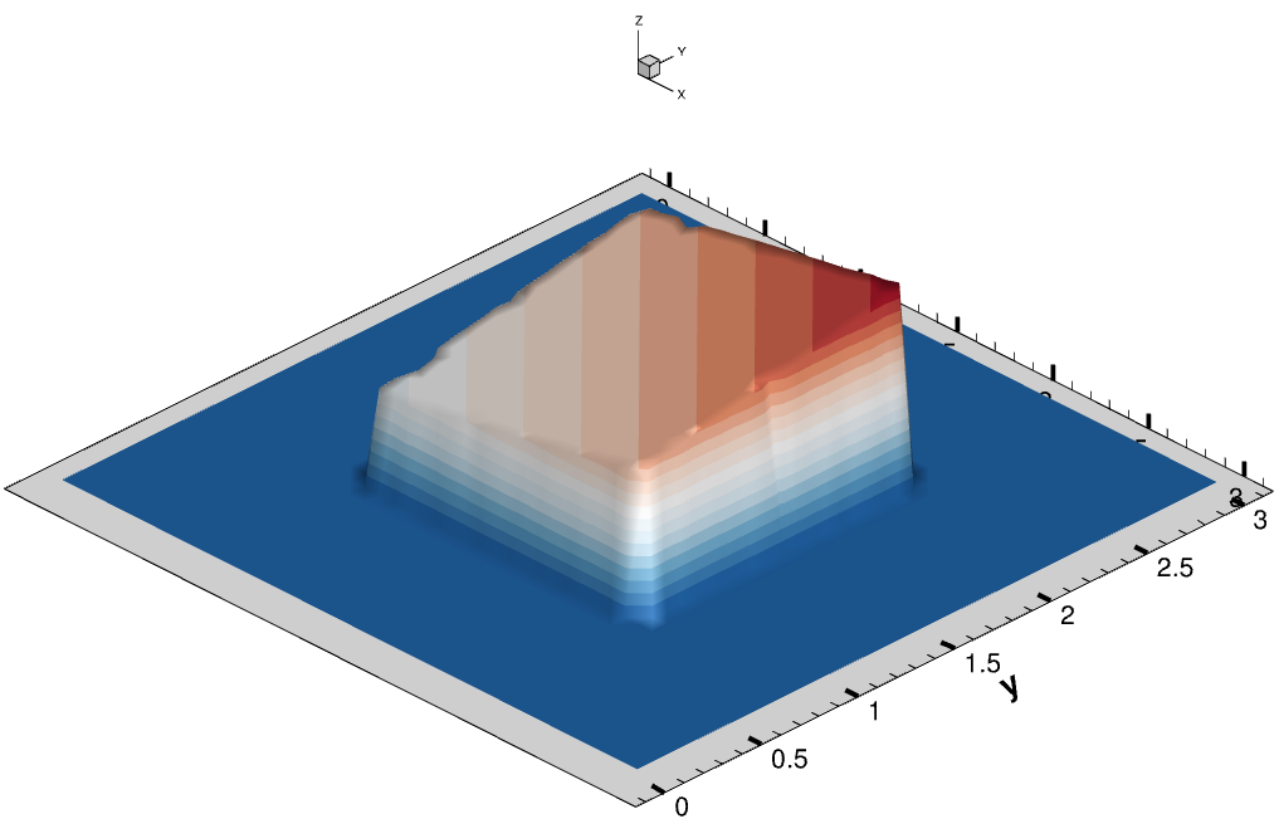} 
\caption{The solution $f_{true}$}
\label{default}
\end{center}
\end{figure}
\begin{figure}[htbp!]
\begin{center}
\includegraphics[scale=0.2]{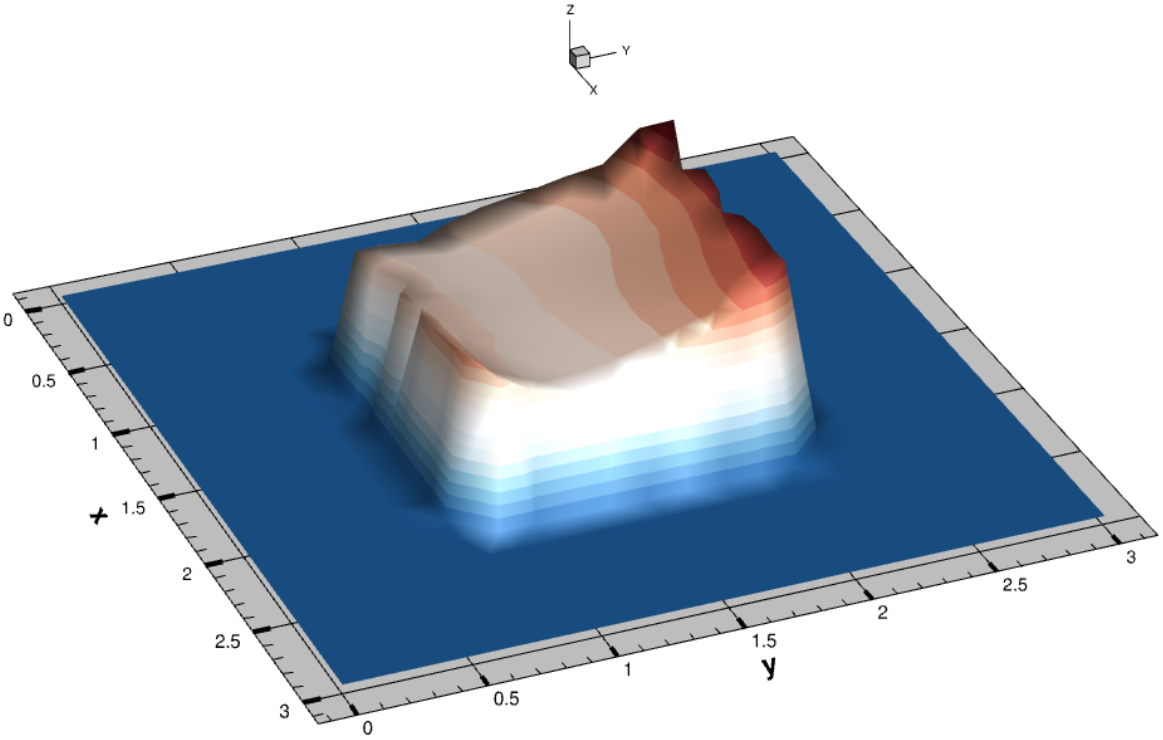} \quad \includegraphics[scale=0.2]{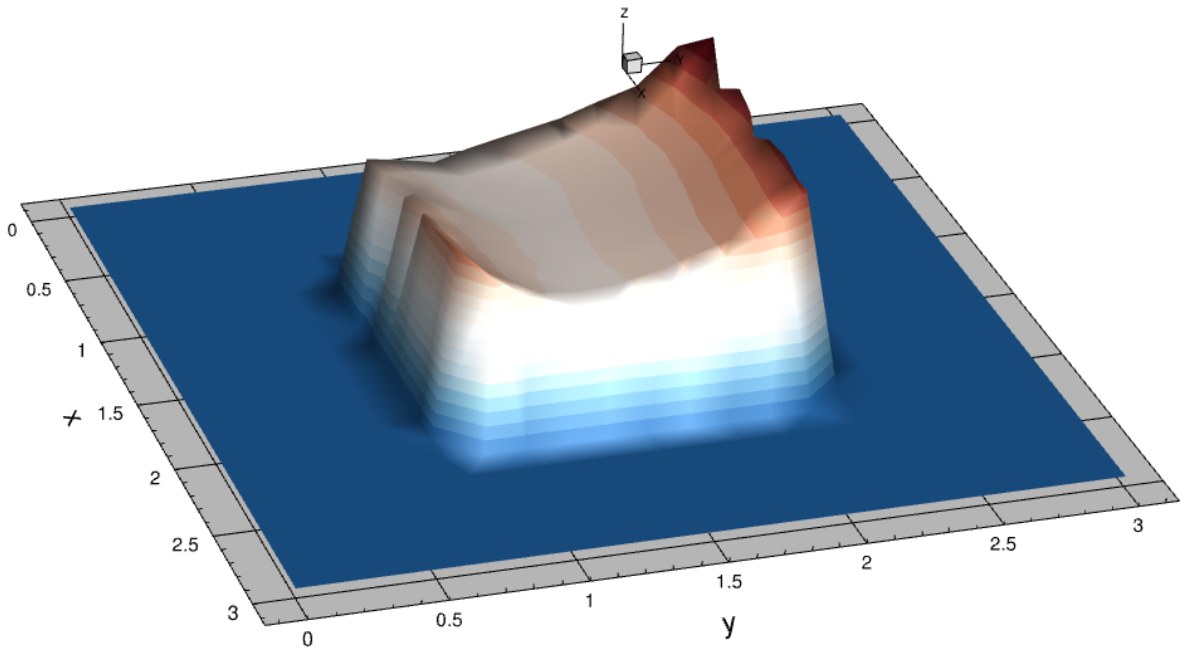} \quad \includegraphics[scale=0.2]{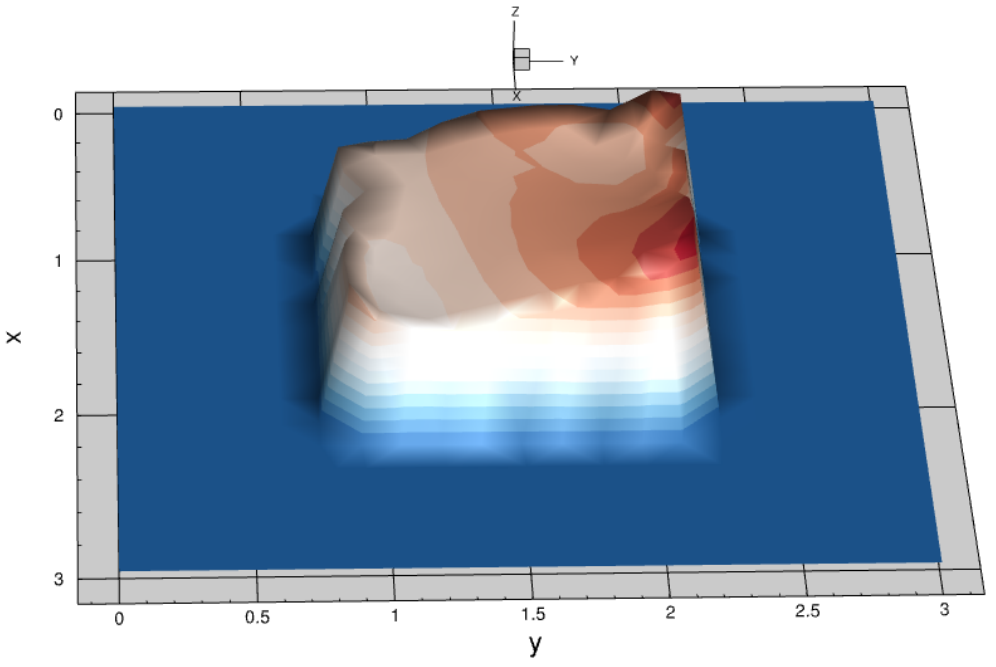} 
\caption{The reconstructions $f_k$ of $f_{true}$ are depicted for $k=10$ (Fig. (a), left),  $k=20$ (Fig. (b), middle) and $k=30$ (Fig. (c),  right), with  the constant $C=0.01$.}
\label{default}
\end{center}
\end{figure}
\subsection{Example 2.} $f_{true}(x)=e^{-((x_1-3/2)^2+(x_2-3/2)^2)}  \chi_\Omega (x)$.
We set a time step $\Delta t=0.07$ and the mesh size $h=0.1$. The support of  $f_{true}$ is a subset of $\bar\Omega=[3/4;9/4]$. The initial guess is $f_0(x)=2\chi_\Omega (x)$.   \par 

\begin{table}[htbp!]
\caption{Parameter setting and the corresponding numerical performances in Case 2.}
\begin{center}
\begin{tabular}{|c|c|c|c|}\hline C & k & err & Illustration  \\
\hline 0.01 & 10 & 22\% & \mbox{Figure (a)} \\
\hline 0.01 & 30 & 18,5\% &  \mbox{Figure (b)}  \\
\hline 0.01 & 30 & 18\% &  \mbox{Figure (c)} \\
\hline
\end{tabular}
\end{center}
\label{Parameter setting and the corresponding numerical performances in Case 2.}
\end{table}

\begin{figure}[htbp!]
\begin{center}
\includegraphics[scale=0.2]{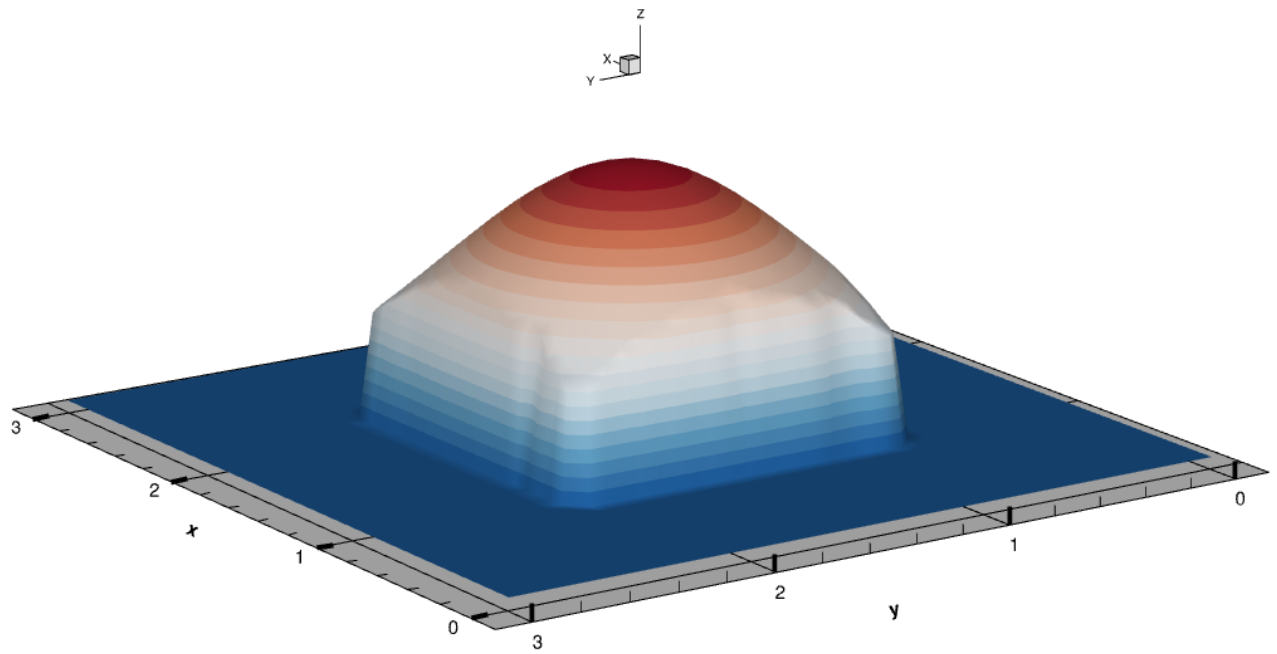}
\caption{The true solution $f_{true}$.}
\label{default}
\end{center}
\end{figure}
\begin{figure}[htbp!]
\begin{center}
\includegraphics[scale=0.2]{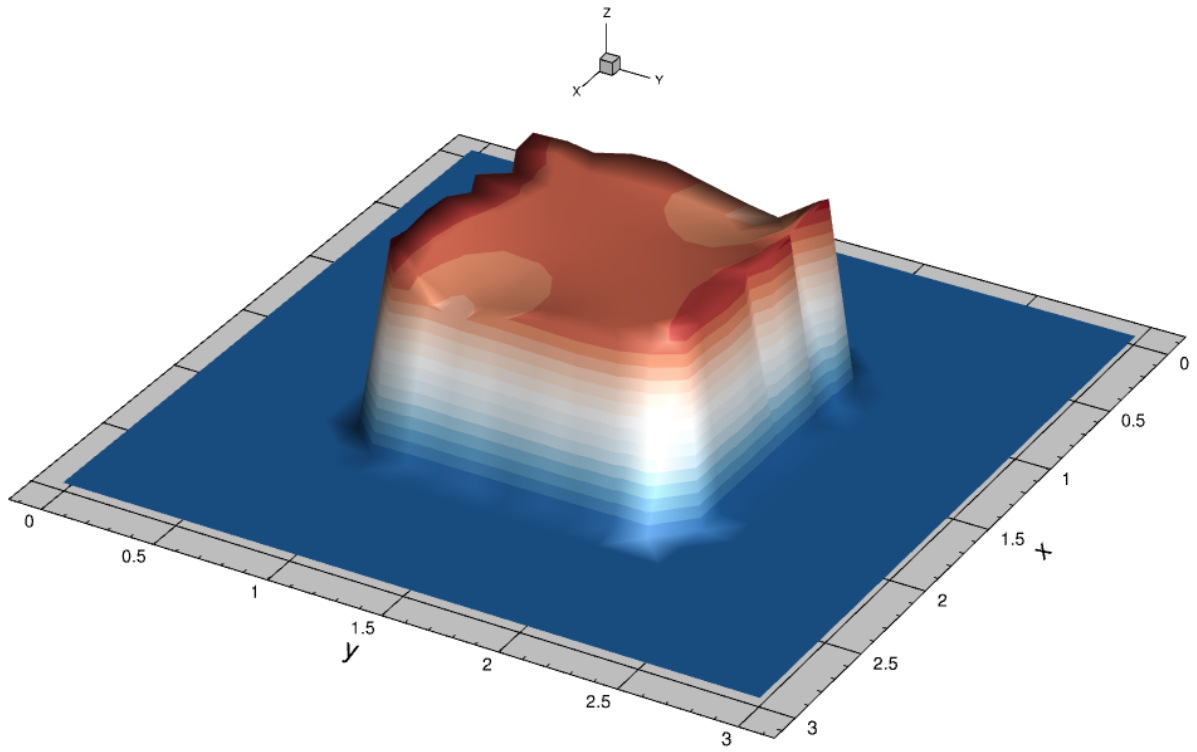} \quad \includegraphics[scale=0.2]{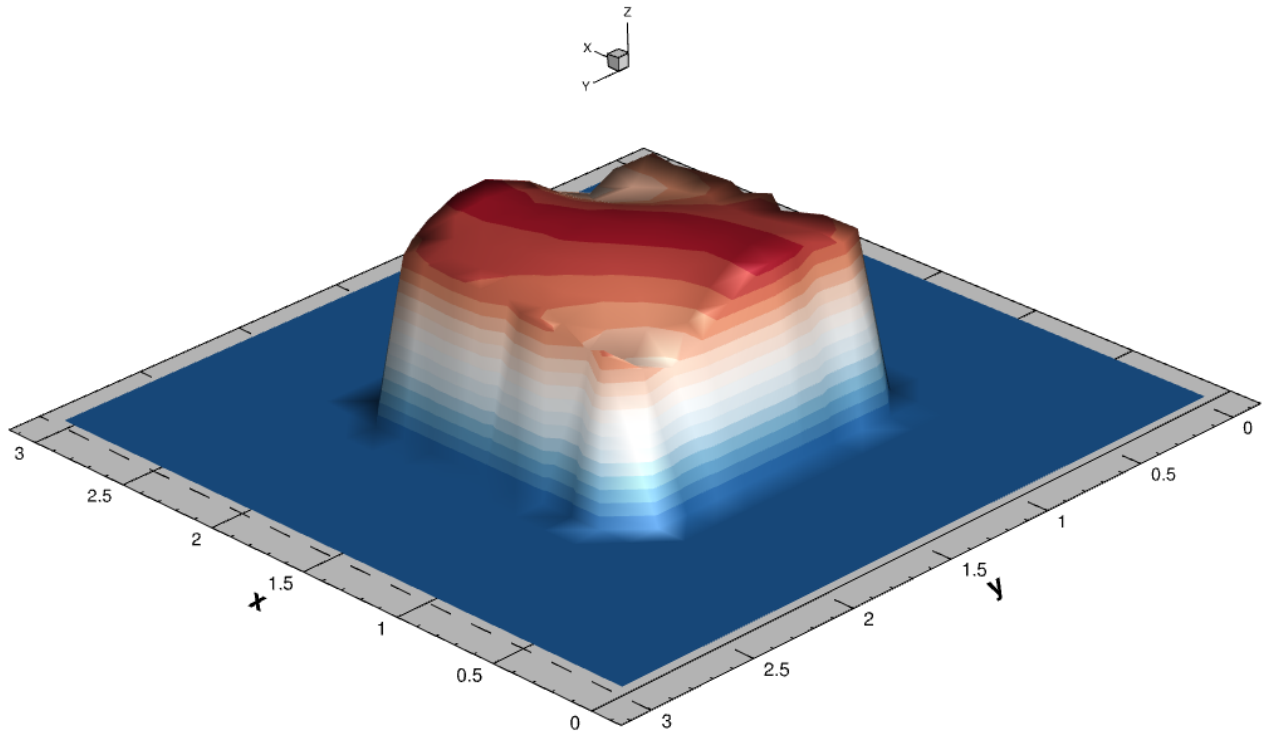} \quad \includegraphics[scale=0.2]{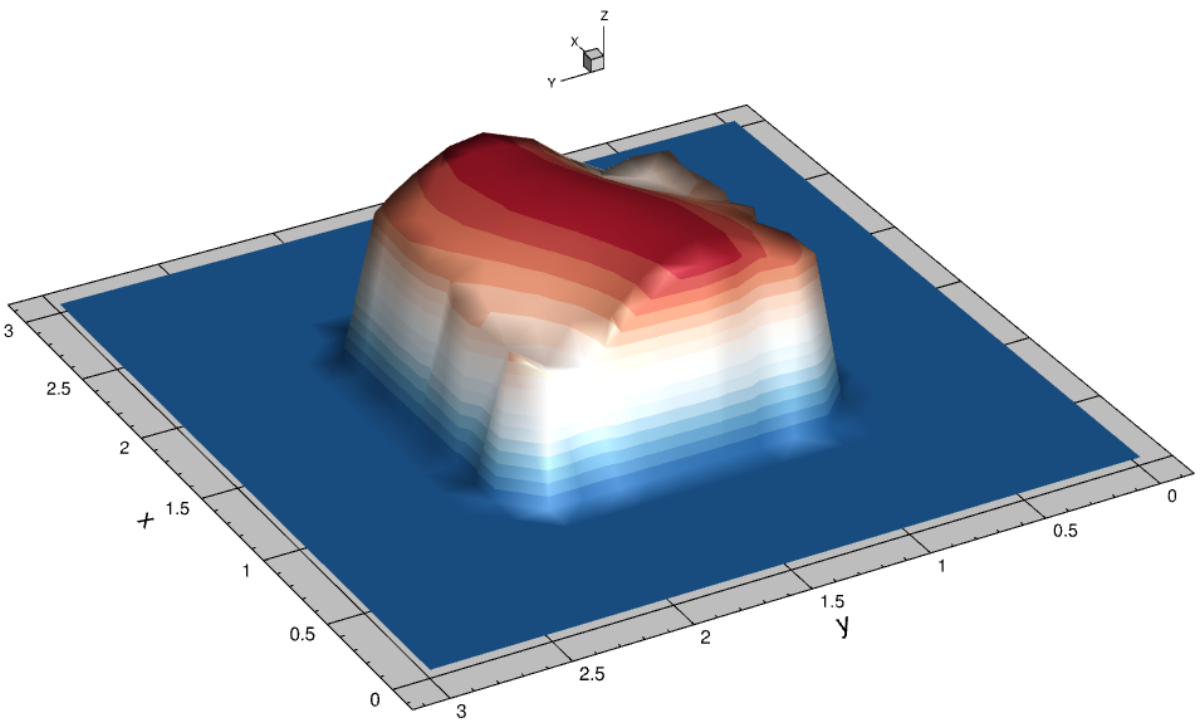} 
\caption{The reconstructions $f_k$ of $f_{true}$ are depicted for $k=10$ (Fig. (a), left),  $k=20$ (Fig. (b), middlle) and $k=30$ (Fig. (c),right), with  the constant $C=0.01$.}
\label{default}
\end{center}
\end{figure}
\subsection{Example  3. } $f_{true}(x)=\cos{{\pi x_1}\over{3}}\cos{{\pi x_2}\over{3}}  \chi_\Omega (x)$.
We set a time step $\Delta t=0.07$ and the mesh size $h=0.1$. The support of  $f_{true}$ is a subset of $\Omega=]3/4;9/4[$.  The initial guess is $f_0(x)=0.5\chi_\Omega (x)$.   \par  
\begin{table}[htp]
\caption{Parameter setting and the corresponding numerical performances in Case 3.}
\begin{center}
\begin{tabular}{|c|c|c|c|}\hline C & k & err & Illustration  \\
\hline 0.01 & 10 & 13\% & \mbox{Figure (a)} \\
\hline 0.01 & 20 & 9,7\% &  \mbox{Figure (b)} \\
\hline 0.01 & 30 & 7\% &  \mbox{Figure (c)} \\
\hline
\end{tabular}
\end{center}
\label{Parameter setting and the corresponding numerical performances in Case 3.}
\end{table}
\begin{figure}[htbp!]
\begin{center}
\includegraphics[scale=0.4]{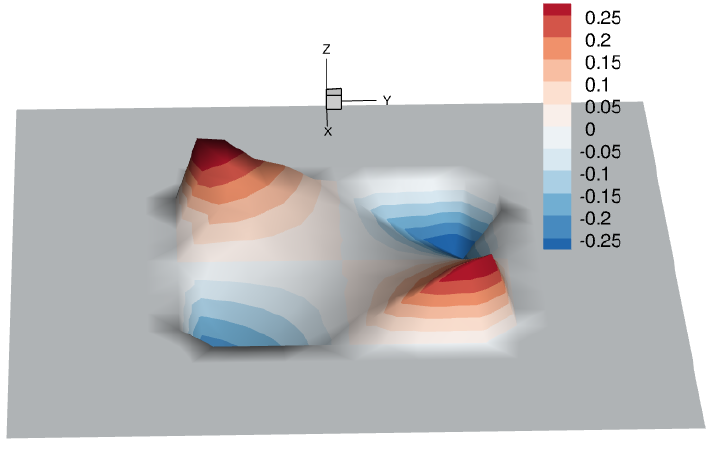} 
\caption{The true solution $f_{true}$}
\label{default}
\end{center}
\end{figure}
 
\begin{figure}[htbp!]
\begin{center}
\includegraphics[scale=0.23]{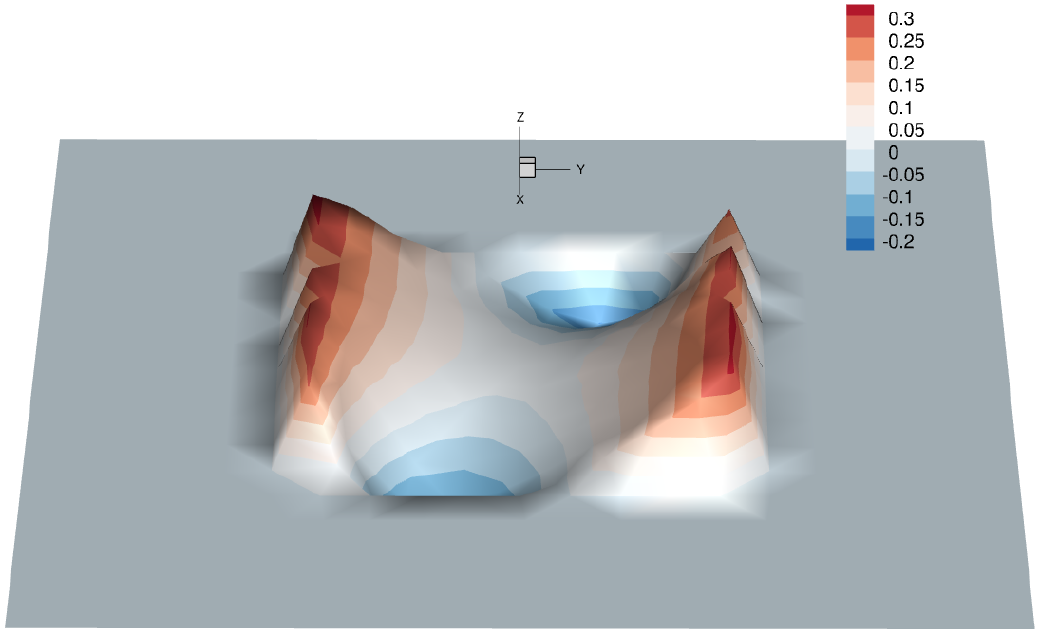} \quad \includegraphics[scale=0.23]{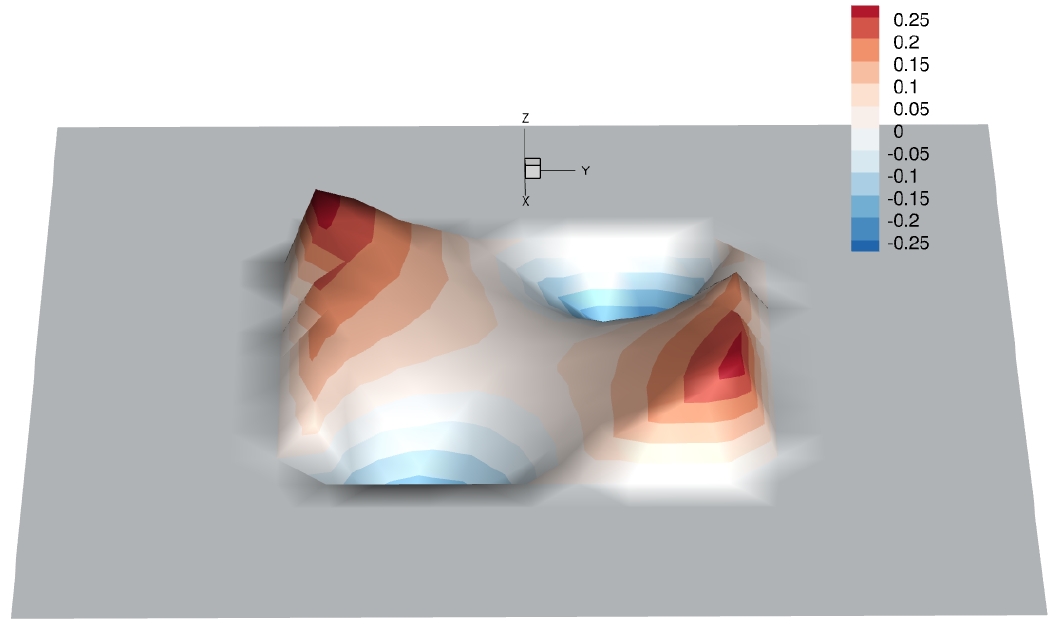} \quad \includegraphics[scale=0.23]{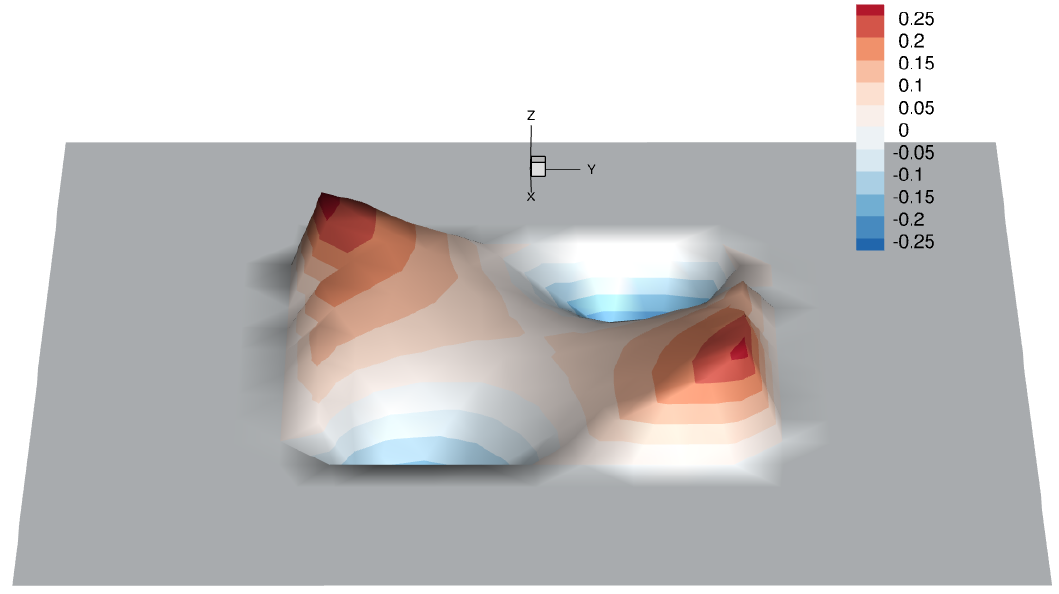} 
\caption{The reconstructions $f_k$ of $f_{true}$ are depicted for $k=10$ (Fig. (a), left),  $k=20$ (Fig. (b), middlle) and $k=30$ (Fig. (c), right), with  the constant $C=0.01$.}
\label{default}
\end{center}
\end{figure}
\subsection{Comments} 
The method was implemented in FreeFem++ (\url{https://freefem.org}), a high level, free software package  managing mesh generation and adaption, matrix assembly, and linear system resolution automatically and generally uses as input the resolution domain, boundary conditions, and bilinear and linear forms.  
Visualization of the three-dimensional results was conducted using Tecplot software. All computations were performed on a 1.4 GHz quad-core Intel Core i5 processor MacBook Pro laptop. \par  
Numerical simulations indicate that the source term can be reconstructed with a relative error ranging from $7\%$ to $18\%$. Furthermore, our observations during the simulations reveal that the iterative reconstruction algorithm (\ref{IterativeReconstEq}) is highly sensitive to the parameter $c$.  This parameter  acts as a scaling factor that balances the influence of the previous iterate $f_k$ and the correction term $\int_0^T\sigma(t)z_{f_k}(\cdot,t) dt$.  As $c$ increases, the weight $\frac{c}{c+\lambda}$ increases, giving more influence to the previous iterate $f_k$ in the next step $f_{k+1}$.  Conversely, as $c$ decreases, the weight $\frac{1}{c+\lambda}$ increases, giving more influence to the correction term.  When $c$ is large relative to $\lambda$, $\frac{c}{c+\lambda} \approx 1$ and $\frac{1}{c+\lambda} \approx \frac{1}{c}$. This makes the update $f_{k+1}$ rely heavily on $f_k$, potentially slowing down convergence since the correction term has less impact.  However, large $c$ can contribute to stability by preventing abrupt changes in the iterates. When $c$ is small relative to $\lambda$, $\frac{c}{c+\lambda} \approx \frac{c}{\lambda}$ and $\frac{1}{c+\lambda} \approx \frac{1}{\lambda}$. This makes the update $f_{k+1}$ rely heavily on the correction term, potentially accelerating convergence.  However, small $c$ can lead to instability if the correction term introduces significant changes at each step.  For our simulations, we empirically chose  $c=0.01$  and $\lambda=10^{-5}$. This choice ensures that: 
$$f_{k+1}|_{\Omega} \approx f_k|_{\Omega} - \frac{1}{c} \int_0^T \sigma(t) z_{f_k}(\cdot, t) |_{\Omega} dt. 
$$
This leads to gradual updates and potentially slower convergence.
\par 
The convergence and stability of the algorithm (\ref{IterativeReconstEq}) rely on the proper choice of $c$. A balance must be struck to ensure the algorithm converges efficiently without becoming unstable. Large $c$ values favor stability but may slow down convergence, while small $c$ values can speed up convergence but risk instability

\section{Appendix}
\subsection{Analytic properties of solutions of Stokes system}
In this subsection we consider analytic properties of the solution $u$ of \eqref{eq1} with $F\equiv0$. More precisely, we show the following.

\begin{proposition}\label{p1} Let $u\in L^2(0,T;\mathbb V)\cap H^1(0,T;\mathbb V')$ be solution  of \eqref{eq1} with $F\equiv0$ and $u_0\in\mathbb H$. Then, for any $t\in(0,T)$, the map $\R^n\ni x\mapsto u(x,t)$ is real analytic.\end{proposition}
\begin{proof} This result is well known but we give its proof for sake of completeness. The solution of \eqref{eq1} is given by  
\begin{equation*}
u(t,x) = \int_{\mathbb{R}^n} \frac{1}{(4\pi\,\nu\,t)^{\frac{n}{2}}} 
  e^{-\frac{|x-x'|^2}{4\nu t}} \, u_0(x') dx'.
\end{equation*}
For a fixed time $t>0$ and setting $z=x+i y  \in \mathbb{C}^n$ 
we consider the $n$ complex variable function 
\begin{equation*}
F(z) = \int_{\mathbb{R}^n} f(z,x') dx' 
  \quad\text{where}\quad 
  f(x+ i y,x') = \frac{1}{(4\pi\,\nu\,t)^{\frac{n}{2}}}  
     e^{-\frac{(x+i y-x')\cdot (x+i y-x')}{4\nu t}} \, u_0(x'),  
\end{equation*}
and it follows $u(t,x) = F(x)$. 
We claim that the integrand $z \to F(z)$ is holomorphic in the whole space $\mathbb{C}^n$, 
implying that the solution $x \to u(t,x)$ is real analytic. 

Let $R>0$.
We verify that for $|z|=(\sum_{k=1}^n |z_k|^2)^{\frac12}\leq R$ 
and $1\leq k \leq n$ 
we have 
\begin{align*}
& \left| f(z,x') \right| \,\leq\, 
  g_0(x') := \frac{1}{(4\pi\,\nu\,t)^{\frac{n}{2}}} 
    e^{\frac{R^2+2R |x'|-|x'|^2}{4\nu t}} \, |u_0(x')|, 
\\
& \left| \frac{\partial}{\partial x_k} f(z,x') \right| ,
  \left| \frac{\partial}{\partial y_k} f(z,x') \right| \,\leq\,
  g_1(x') :=  \frac{2(R+|x'|)}{4\nu t} \, g_0(x').
\end{align*}
Since the functions $g_0$ and $g_1$ are integrable
and the function $z \to f(z,x')$ is holomorphic 
in the ball $|z|<R$  and verifies the Cauchy-Riemann equations in each variable separately,  
we can therefore conclude from the Dominated Convergence Theorem that 
the function $F(z)$ is continuously differentiable. 
Moreover it also satisfies the Cauchy-Riemann equations in each variable separately :  
As a result, $F$ is also holomorphic in the ball $|z|<R$ for all $R>0$, 
and consequently in the whole space $\mathbb{C}^n$. 
This concludes the proof of the proposition.  \end{proof}

\subsection{Counterexamples for the determination of general time dependent source term}

In this subsection we recall a counterexample to the  determination of general class of time dependent source terms. For this purpose we fix $\chi\in C^\infty_0(\Omega\times(0,T))^n$ such that $\nabla\times \chi\not\equiv0$ and choose $F_1=\partial_t(\nabla\times \chi)$, $F_2=\nu\Delta (\nabla\times \chi)$. Here $\nabla\times$ denotes the rotational operator on $\R^n$. Then, for $j=1,2$, fixing $u_j$ the solution of \eqref{eq1} with $F=F_j$ and considering $u=u_1-u_2$ we see that $u$ solves \eqref{eq1} with $u_0\equiv0$ and
$$F=\partial_t(\nabla\times \chi) -\nu\Delta (\nabla\times \chi).$$
One can  easily check that $\nabla\times \chi$ is the unique solution of \eqref{eq1} and therefore $u=\nabla\times \chi$. 
Thus, since $\chi$ is equal to 0 in a neighborhood of $\partial\Omega$ and $\partial\mathcal{O}$, we get 
$$u(x,t)=\nabla\times \chi(x,t)=0,\quad (x,t)\in(\partial\Omega\times(0,T))\cup(\partial\mathcal O\times(0,T)).$$
However, since $\nabla\times \chi\not\equiv0$, we deduce that $F\not\equiv0$ since otherwise by the uniqueness of the solution of \eqref{eq1} we would have $\nabla\times \chi=u\equiv0$ which contradicts the assumption imposed to the function $\chi$. This show that we have
$$(u_1(x,t)=u_2(x,t),\quad (x,t)\in(\partial\Omega\times(0,T))\cup(\partial\mathcal O\times(0,T)))
\not\Longrightarrow (F_1=F_2)$$
which proves that it is impossible to determine general class of time dependent source terms.
\subsection{Counterexamples for the determination of  time dependent source term with separated variables}
Let us observe that  counterexamples of the preceding section can also be extended to general source terms 
with separated variables of the form \eqref{source1} and \eqref{source2}. 

$\bullet$ {\bf Counterexample for source terms 
with separated variables of the form \eqref{source1}}: 
We choose $\psi\in C^\infty_0(\Omega)^n$ such that $\nabla\times \psi\not\equiv0$, $\beta\in H^1_0(0,T)$, with $\sigma\not\equiv0$, and choose $F_j$, $j=1,2$, of the form \eqref{source1} with $\sigma=\sigma_j$ and $f=f_j$ given by
$$f_1(x)=\nabla\times \psi(x),\ \sigma_1(t)=\beta'(t),\ \sigma_2(t)=\beta(t),\ f_2(x)=\nu\Delta (\nabla\times \psi)(x),\quad x\in\R^n,\ t\in(0,T).$$
Then, for $j=1,2$, fixing $u_j$ the solution of \eqref{eq1} with $F=F_j$ and considering $u=u_1-u_2$ we see that $u$ solves \eqref{eq1} with $u_0\equiv0$ and
$$F(x,t)=\partial_t(\beta(t)\nabla\times \psi(x))-\nu\Delta(\beta(t) \nabla\times \psi(x)),\quad x\in\R^n,\ t\in(0,T).$$
Repeating the arguments of the previous section, we see that we have
$$u(x,t)=\beta(t)\nabla\times \psi(x),\quad x\in\R^n,\ t\in(0,T)$$
and $F\not\equiv0$. In particular, we have
$$u(x,t)=\beta(t)\nabla\times \psi(x)=0,\quad (x,t)\in(\partial\Omega\times(0,T))\cup(\partial\mathcal O\times(0,T)).$$
which proves that the condition \eqref{t1a} will be fulfilled but \eqref{t1b} is not true. This show that we have
$$\begin{aligned}&u_1(x,t)=u_2(x,t),\  (x,t)\in(\partial\Omega\times(0,T))\cup(\partial\mathcal O\times(0,T))\\
&\not\Longrightarrow (\sigma_1(t) f_1(x)=\sigma_2(t)f_2(x),\ x\in\R^n,\ t\in(0,T))\end{aligned}$$
which proves that it is impossible to determine general source terms with separated variables of the form \eqref{source1}

$\bullet$ {\bf Counterexample for source terms 
with separated variables of the form \eqref{source2}}: 
Let us consider $\chi\in C^\infty_0(\R^{n-1}\times(0,T))^{n-1}$ such that $\nabla'\times \chi\not\equiv0$, where $\nabla'\times$ denotes the rotational operator on $\R^{n-1}$. Choose also $g\in C^\infty_0(\R)$ such that $g\not\equiv0$ and such that for all $t\in[0,T]$, the support of the map $x=(x',x_n)\mapsto \chi(x',t)g(x_n)$ is contained into  $\Omega$. Now  choose $F_j$, $j=1,2$, of the form \eqref{source2} with $G=G_j$ and $h=h_j$ given by
$$G_1(x',t)=\partial_t\nabla'\times \chi(x',t)-\nu\Delta'\nabla'\times \chi(x',t),\ h_1(x_n)=g(x_n),\quad x'\in\R^{n-1},\ x_n\in\R,\ t\in(0,T),$$
$$ G_2(x',t)=\nabla'\times \chi(x',t),\ h_2(x_n)=\nu \partial_{x_n}^2g(x_n),\quad x'\in\R^{n-1},\ x_n\in\R,\ t\in(0,T),$$
where 
$$\Delta'=\sum_{j=1}^{n-1}\partial_{x_j}^2.$$
Then, for $j=1,2$, fixing $u_j$ the solution of \eqref{eq1} with $F=F_j$ and considering $u=u_1-u_2$ we see that $u$ solves \eqref{eq1} with $u_0\equiv0$ and
$$F(x,t)=\left([\partial_t(g(x_n)\nabla'\times \chi(x',t))-\nu\Delta(g(x_n)\nabla'\times \chi(x',t))]^T,0\right)^T,\quad x'\in\R^{n-1},\ x_n\in\R,\ t\in(0,T).$$
Repeating the above arguments, we see that 
$$u(x',x_n,t)= \left(g(x_n)\nabla'\times \chi(x',t)^T,0 \right)^T,\quad x'\in\R^{n-1},\ x_n\in\R,\ t\in(0,T)$$
and $F\not\equiv0$. In particular, we have
$$u(x',x_n,t)=\left(g(x_n)\nabla'\times \chi(x',t)^T,0\right)^T=0,\quad (x',x_n,t)\in(\partial\Omega\times(0,T))\cup(\partial\mathcal O\times(0,T)),$$
which proves that the condition \eqref{t1a} will be fulfilled but \eqref{t1b} is not true. This show that we have
$$\begin{aligned}&u_1(x,t)=u_2(x,t),\  (x,t)\in(\partial\Omega\times(0,T))\cup(\partial\mathcal O\times(0,T))\\
&\not\Longrightarrow (G_1(x',t)h_1(x_n)=G_2(x',t)h_2(x_n),\ x'\in\R^{n-1},\ x_n\in\R,\ t\in(0,T))\end{aligned}$$
which proves that it is impossible to determine general source terms with separated variables of the form \eqref{source2}

\end{document}